\documentclass[11pt,reqno]{amsart}
\usepackage{eucal,fullpage,times,amsmath,amsthm,amssymb,mathrsfs,stmaryrd,color,enumerate,accents}
\usepackage[all]{xy}
\usepackage{url}
\usepackage[colorlinks=true,hyperindex]{hyperref}

\newcommand{\calO}{{\mathcal{O}}}
\newcommand{\calA}{{\mathcal{A}}}
\newcommand{\calB}{{\mathcal{B}}}
\newcommand{\calC}{\mathcal{C}}

\newcommand{\calD}{\mathcal{D}}
\newcommand{\calH}{\mathcal{H}}
\newcommand{\calF}{\mathcal{F}}

\newcommand{\Z}{\mathbf{Z}}
\newcommand{\G}{\mathbf{G}}
\newcommand{\N}{\mathbf{N}}
\newcommand{\C}{\mathbf{C}}
\newcommand{\F}{\mathbf{F}}
\newcommand{\Q}{\mathbf{Q}}
\newcommand{\A}{\mathbf{A}}

\renewcommand{\P}{\mathbf{P}}

\newcommand{\Spec}{{\mathrm{Spec}}}
\newcommand{\Aff}{\mathrm{Aff}}
\newcommand{\QAff}{\mathrm{QAff}}

\newcommand{\CAlg}{\mathrm{CAlg}}

\newcommand{\Ab}{\mathrm{Ab}}
\newcommand{\Hom}{\mathrm{Hom}}
\newcommand{\Fun}{\mathrm{Fun}}
\newcommand{\Tot}{\mathrm{Tot}}

\newcommand{\Ext}{\mathrm{Ext}}
\newcommand{\Tor}{\mathrm{Tor}}
\newcommand{\QCoh}{\mathrm{QCoh}}
\newcommand{\Ind}{\mathrm{Ind}}
\newcommand{\perf}{\mathrm{perf}}
\newcommand{\Coh}{\mathrm{Coh}}

\newcommand{\Vect}{\mathrm{Vect}}

\newcommand{\Mod}{\mathrm{Mod}}

\newcommand{\Pic}{\mathrm{Pic}}

\newcommand{\et}{\mathrm{\acute{e}t}}

\newcommand{\im}{\mathrm{im}}
\newcommand{\id}{\mathrm{id}}
\newcommand{\coker}{\mathrm{coker}}
\newcommand{\cofib}{\mathrm{cofib}}
\newcommand{\fib}{\mathrm{fib}}

\renewcommand{\ker}{\mathrm{ker}}

\newcommand{\opp}{\mathrm{opp}}

\newcommand{\colim}{\mathop{\mathrm{colim}}}

\newcommand{\adjunction}[4]{\xymatrix@1{#1{\ } \ar@<0.3ex>[r]^{ {\scriptstyle #2}} & {\ } #3 \ar@<0.3ex>[l]^{ {\scriptstyle #4}}}}

\begin{document}

\bibliographystyle{alpha}

\newtheorem{theorem}{Theorem}[section]
\newtheorem*{theorem*}{Theorem}
\newtheorem*{definition*}{Definition}
\newtheorem{proposition}[theorem]{Proposition}
\newtheorem{lemma}[theorem]{Lemma}
\newtheorem{corollary}[theorem]{Corollary}

\theoremstyle{definition}
\newtheorem{definition}[theorem]{Definition}
\newtheorem{question}[theorem]{Question}
\newtheorem{remark}[theorem]{Remark}
\newtheorem{example}[theorem]{Example}
\newtheorem{notation}[theorem]{Notation}
\newtheorem{convention}[theorem]{Convention}
\newtheorem{construction}[theorem]{Construction}
\newtheorem{claim}[theorem]{Claim}

\title{Algebraization and Tannaka duality}
\address{Institute for Advanced Study \\ Einstein Drive \\ Princeton, NJ 08540}
\email{bhargav.bhatt@gmail.com}
\author{Bhargav Bhatt}
\maketitle

\section{Introduction}

Our goal in this paper is to identify certain naturally occurring colimits of schemes and algebraic spaces. The statements, which are essentially algebraization results for maps between schemes and algebraic spaces, are elementary and explicit. However, our techniques are indirect: we use (and prove) some new Tannaka duality theorems for maps of algebraic spaces. Our approach to these theorems relies on a systematic deployment of perfect complexes (ergo, we use some derived algebraic geometry) instead of ample line bundles or vector bundles.  Consequently, the Tannaka duality results we obtain have fewer, and much weaker, finiteness constraints than some of the existing ones: we only insist that our algebraic spaces be quasi-compact and quasi-separated (qcqs), and do not require any quasi-projectivity or noetherian hypotheses. All rings are assumed to be commutative.

\subsection{Algebraization of jets}
The first colimit we identify is that of an affine (adic) formal scheme.

\begin{theorem}
	\label{thm:alglim}
If $A$ is a ring which is $I$-adically complete for some ideal $I$, and $X$ is a qcqs algebraic space, then $X(A) \simeq \lim X(A/I^n)$ via the natural map.
\end{theorem}

An equivalent formulation is: if $A = \lim A/I^n$, then $\Spec(A)$ is a colimit of the diagram $\{\Spec(A/I^n)\}$ in the category of qcqs algebraic spaces. Theorem \ref{thm:alglim} is straightforward to prove if $A/I$ is local; its content becomes apparent only when $\Spec(A/I)$ has some non-trivial global geometry. Note also that there are no noetherian assumptions on any object in sight, so the ideal $I$ might not be finitely generated. In fact, the result extends to more general topological rings $A$ that arise naturally in $p$-adic geometry (see Remark \ref{rmk:alglimadm}). This answers a question asked by Drinfeld, and has the following representability consequence in the theory of arc spaces, which was our original motivation for pursuing Theorem \ref{thm:alglim}.

\begin{corollary}
	\label{cor:formalarcdescent}
	If $X$ is a qcqs algebraic space, then the ``formal arc'' functor $\mathrm{Arcs}_X(R) := X(R\llbracket t \rrbracket)$ is an fpqc sheaf on the category of rings, and is identified with the functor $R \mapsto \lim X(R[t]/(t^n))$.
\end{corollary}

As the functor $\mathrm{Arcs}_X$ is almost never locally finitely presented (even for $X$ an algebraic variety),  one cannot reduce Corollary \ref{cor:formalarcdescent} to the corresponding assertion on the category of noetherian rings (which is easier to prove). This corollary answers a question raised in \cite[\S 2]{NicaiseSebag} and pointed out to us by Nicaise. The following feature of the proof of Theorem \ref{thm:alglim} seems noteworthy: given a compatible system $\{\epsilon_n:\Spec(A/I^n) \to X\} \in \lim X(A/I^n)$, we construct an algebraization $\epsilon:\Spec(A) \to X$ without ever musing about points of $\Spec(A) \setminus \Spec(A/I)$.

\subsection{Algebraization of products}
The second result deals with products, rather than cofiltered inverse limits, of rings; this question was brought to our attention by Poonen.

\begin{theorem}
	\label{thm:algprod}
	If $\{A_i\}_{i \in I}$ is a set of rings, and $X$ is a qcqs algebraic space, then $X(\prod_i A_i) \simeq \prod_i X(A_i)$ via the natural map.
\end{theorem}

An equivalent formulation is: the scheme $\Spec(\prod_i A_i)$ is a coproduct of $\{\Spec(A_i)\}$ in the category of qcqs algebraic spaces. Note that {\em some} finiteness hypothesis on $X$ is necessary: the (typically non-quasi-compact) scheme $\sqcup_i \Spec(A_i)$ is a coproduct of $\{\Spec(A_i)\}$ in the category of all schemes. The non-trivial case, again, is when the rings $A_i$ have interesting global geometry. Moreover, like Theorem \ref{thm:alglim}, the proof of Theorem \ref{thm:algprod} also circumvents ever contemplating points of $\Spec(\prod_i A_i) \setminus \sqcup_i \Spec(A_i)$.  Theorem \ref{thm:alglim} may be used to describe adelic points on algebraic spaces over global fields (see Corollary \ref{cor:adelicpoints}).

\subsection{Formal glueing} 
The third result concerns the classical Beauville-Laszlo theorem \cite{BLDescent}, which is incredibly useful in the construction of bundles on families of curves arising, for example, in geometric representation theory. Recall that this theorem asserts: given an affine scheme $X$ with a Cartier divisor $Z \subset X$, one can patch compatible quasi-coherent {\em sheaves} on $\widehat{X}$ (the completion of $X$ along $Z$) and $U := X \setminus Z$ to a quasi-coherent sheaf on $X$, provided the sheaves being patched are flat along $Z$. This is a sheaf-theoretic manifestation of the principle that $\widehat{X}$ is an algebro-geometric analogue of a tubular neighbourhood of $Z$ in $X$, so $X$ behaves as though it were built by glueing $\widehat{X}$ to $U$ over $\widehat{X} \setminus Z = \widehat{X} \times_X U$. In the next theorem, we vivify this geometric intuition by showing that $X$ is literally the pushout of $\widehat{X}$ and $U$ along $\widehat{X} \times_X U$, which perhaps clarifies the glueing result for sheaves. Along the way, we also offer an improvement on the glueing result itself: the patching works unconditionally for quasi-coherent {\em complexes}.

\begin{theorem}
	\label{thm:BLintro}
	Let $\pi:Y \to X$ be a map of qcqs algebraic spaces. Assume there exists a finitely presented closed subspace $Z \subset X$ satisfying\footnote{The condition $Z \times_X^L Y \simeq Z$ means:  the maps $Y \to X$ and $Z \to X$ are mutually $\Tor$-independent, and $\pi^{-1}(Z) \simeq Z$.}  $Z \times_X^L Y \simeq Z$. Set $U = X\setminus Z$ and $V = Y \setminus \pi^{-1}(Z)$. Then:
	\begin{enumerate}
		\item The fibre square 
			\[ \xymatrix{ V \ar[r]^j \ar[d]^\pi & Y \ar[d]^\pi \\
				U \ar[r]^j & X }\]
		 	is a pushout in qcqs algebraic spaces. 
		\item The natural map induces an equivalence $D(X) \simeq D(Y) \times_{D(V)} D(U)$.
	\end{enumerate}
\end{theorem}

Here $D(X)$ is the $\infty$-category of quasi-coherent complexes on $X$. The Beauville-Laszlo theorem concerns the special case where $X = \Spec(A)$ for some ring $A$, $Z = \Spec(A/f)$ for $f \in A$ a regular element, and $Y = \Spec(\lim A/f^n)$ is the completion. In this case, they show an analogue of (2) for modules that are $f$-regular. In order to get the general consequence (2) above, it is crucial to work with $\infty$-categories:  the corresponding statement about the full module category or the classical derived category is false.

\subsection{Tannaka duality}

The ``surjectivity'' assertions in Theorem \ref{thm:alglim} and Theorem \ref{thm:algprod}, as well as (2) in Theorem \ref{thm:BLintro}, may be viewed as algebraization results for maps.  Despite the elementary formulations, we do not have a constructive proof of any of these, even for schemes, except in special cases. Instead, the desired algebraization is constructed by first building a suitable functor on (derived) categories of quasi-coherent sheaves; Tannaka duality results then show that this functor is the pullback functor for a morphism. 

The implementation of the strategy above necessitates certain {\em derived} Tannaka duality results.\footnote{It is easy to see why the derived setting is preferable in approaching Theorem \ref{thm:alglim}: perfect complexes are easier to manipulate than (finitely presented) quasi-coherent sheaves, especially with respect to operations involving both limits and tensor products.} These duality results rely crucially on Lurie's \cite{LurieDAGVIII}, but cannot directly be deduced from it: Lurie works in greater generality, and consequently has stronger hypotheses. Nevertheless, leveraging his ideas  with some more classical techniques,  we show the following, which suffices for the applications above.

\begin{theorem}
	\label{thm:TannakaAlgSpaces}
	If $X$ and $S$ are algebraic spaces with $X$ qcqs, then pullback induces equivalences
	\[ \Hom(S,X) \simeq \Fun_{\otimes}(D_\perf(X),D_\perf(S)) \simeq \Fun^L_{\otimes}(D(X),D(S)).\]
\end{theorem}

Here $D_\perf(X) \subset D(X)$ is the full subcategory of perfect complexes on $X$,  and similarly for $S$. Also, $\Fun_{\otimes}(D_\perf(X),D_\perf(S))$ parametrizes exact symmetric monoidal functors $D_\perf(X) \to D_\perf(S)$, while $\Fun_{\otimes}^L(D(X),D(S))$ parametrizes cocontinuous (i.e., colimit-preserving) symmetric monoidal functors. Note that, due to the existence of Fourier-Mukai transforms, there is no hope of proving such a result without keeping track of the $\otimes$-structure. The relevance of Theorem \ref{thm:TannakaAlgSpaces} to the previous discussion on colimits is:

\begin{corollary}
	\label{cor:colimitschemescrit}
	Fix a qcqs algebraic space $X$, and a diagram  $\{X_i\}$ of qcqs spaces over $X$. If pullback induces $D_\perf(X) \simeq \lim D_\perf(X_i)$, then $X \simeq \colim X_i$ in the category of qcqs algebraic spaces.
\end{corollary}

Besides the applications above, Corollary \ref{cor:colimitschemescrit} should be also useful in excising hypotheses on the diagonal in certain existence results; for example, we indicate in Remark \ref{rmk:chow} why the separatedness assumption in Grothendieck's formal geometry version of Chow's theorem can be dropped completely.

As mentioned above, Lurie proved a related Tannakian result for a very general class of (spectral) derived stacks in \cite[Theorem 3.4.2]{LurieDAGVIII}. When specialized to algebraic spaces, his result differs from Theorem \ref{thm:TannakaAlgSpaces} in two ways: he requires the diagonal of $X$ to be affine, and he ``only'' shows that cocontinuous symmetric monoidal functors $F:D(X) \to D(S)$ that preserve connective objects and {\em flat objects} come from geometry. The first restriction is relatively mild, at least in applications, but the last one is severe, rendering his result inapplicable to Theorems \ref{thm:alglim}, \ref{thm:algprod}, and \ref{thm:BLintro} (as it is quite difficult to control flatness properties of modules through limits).  An analogue of Theorem \ref{thm:TannakaAlgSpaces} for noetherian stacks with some tameness and quasi-projectivity hypotheses (over a field) can be found in \cite{IwanariFukuyama}. A generalization of Theorem \ref{thm:TannakaAlgSpaces} to a fairly large class of stacks, together with some applications, is also the subject of forthcoming joint work of Daniel Halpern-Leistner and the author.

In the world of schemes, one can go further than Theorem \ref{thm:TannakaAlgSpaces} to get an underived statement. In fact, Lurie already did so \cite{LurieTD} for algebraic stacks under the afore-mentioned constraints, and these were removed by Brandenburg and Chirvasitu in the case of schemes to show:

\begin{theorem}\cite{BrandenburgChirvasitu}
	\label{thm:BC}
If $X$ and $S$ are schemes with $X$ qcqs, then pullback induces an equivalence
\[ \Hom(S,X) \simeq \Fun^L_{\otimes}(\QCoh(X),\QCoh(S)).\]
\end{theorem}

Here  $\Fun^L_{\otimes}(\QCoh(X),\QCoh(S))$ denotes the category of all cocontinuous symmetric monoidal functors between the abelian categories of quasi-coherent sheaves on $X$ and $S$.  We can use this result, in lieu of Theorem \ref{thm:TannakaAlgSpaces}, to prove Theorem \ref{thm:algprod} in the world of schemes. We conclude this introduction by recording a strengthening of Theorem \ref{thm:BC} in a special case that arises often in practice.

\begin{proposition}
	\label{prop:BCforVect}
	Fix schemes $X$ and $S$. If $X$ is qcqs with enough vector bundles, then pullback induces 
\[ \Hom(S,X) \simeq \Fun^L_{\otimes}(\Vect(X),\Vect(S)).\]
\end{proposition}

Here the assumption on $X$ means that every finitely presented quasi-coherent sheaf is the cokernel of a map of vector bundles; any scheme that is quasi-projective over an affine has this property, and this property is studied in depth in \cite{TotaroResProperty}. The object $\Fun^L_{\otimes}(\Vect(X),\Vect(S))$ denotes the category of right exact symmetric monoidal functors $\Vect(X) \to \Vect(S)$. It is important to note that the property of being ``right exact'' for a sequence of bundles is not intrinsic to the category of vector bundles: one needs the ambient category of all quasi-coherent sheaves to make sense of it. The quasi-projective case of Proposition \ref{prop:BCforVect} was also shown much earlier by Savin \cite{SavinVectDuality} with a different proof.

\subsection*{Sketch of proofs}
We begin with Theorem \ref{thm:alglim}.  The injectivity of $X(A) \to \lim X(A/I^n)$ is relatively elementary. For surjectivity, given compatible maps $\{\epsilon_n:\Spec(A/I^n) \to X\}_{n \in \N}$, one must construct a map $\epsilon:\Spec(A) \to X$ algebraizing $\{\epsilon_n\}$.  If $X$ is a quasi-projective variety, then $X$ has ``enough'' vector bundles: every quasi-coherent sheaf can be ``approximated'' by finite complexes of vector bundles. The data $\{\epsilon_n\}$ defines a functor $F:\Vect(X) \to \Vect(\Spec(A))$ as the composition of the pullback $\lim \epsilon_n^*:\Vect(X) \to \lim \Vect(\Spec(A/I^n))$ and the inverse of the equivalence $\Vect(\Spec(A)) \simeq \lim \Vect(\Spec(A/I^n))$ (see Lemma \ref{lem:vectcont}). One then checks that $F$ preserves exact sequences, so Proposition \ref{prop:BCforVect} gives the desired map $\epsilon:\Spec(A) \to X$. In general, $X$ might not admit a single non-trivial vector bundle, which renders this approach useless. However, $X$ always has enough perfect complexes by a fundamental result going back to Thomason \cite{ThomasonTrobaugh}. Hence, the preceding strategy can be salvaged at the derived level using perfect complexes, instead of vector bundles, and Theorem \ref{thm:TannakaAlgSpaces}.

The proof of Theorem \ref{thm:algprod} is similar, but the construction of an appropriate pullback functor $D_{\perf}(X) \to D_{\perf}(\Spec(\prod_i A_i))$ associated to a family of maps $\epsilon_i:\Spec(A_i) \to X$ is harder: one must show that if $K \in D_{\perf}(X)$, then $\prod_i f_i^* K$ is a perfect $(\prod_i A_i)$-complex. We check this by verifying that the ``size'' of $f_i^* K$ is bounded independently of $i$, which, in turn, is accomplished via an analysis of the number of sections needed to generate a module over a ring once the corresponding numbers over a Nisnevich cover have been specified. More details can be found at the start of \S \ref{sec:algprodspaces}. 

For Theorem \ref{thm:TannakaAlgSpaces}, the full faithfulness of $\Hom(S,X) \to \Fun^L_{\otimes}(D(X),D(S))$ is a consequence of a result of Lurie. For essential surjectivity, fix a cocontinuous symmetric monoidal functor $F:D(X) \to D(S)$. We first check that $F$ preserves connecitivty; this allows us to ``pull back'' affine $X$-spaces to affine $S$-spaces simply by applying $F$ to the corresponding commutative algebra in $D^{\leq 0}(X)$. Viewing a quasi-affine $X$-space as the complement of a (constructible) closed subspace in an affine $X$-space, one may also pull back quasi-affine $X$-spaces along $F$. The crucial assertion is that this procedure respects \'etale morphisms as well as coverings; this is deduced by showing the analogous assertions for arbitrary commutative algebras in $D(X)$ as pushing forward the structure sheaf gives a fully faithful embedding of quasi-affine $X$-spaces into commutative algebras in $D(X)$. Consequently, the construction of $f:S \to X$ such that $f^* = F$ is \'etale local on $X$, so we reduce to the case where $X$ is affine, which is easy.

Proposition \ref{prop:BCforVect} is deduced painlessly from Theorem \ref{thm:BC} by writing quasi-coherent sheaves as filtered colimits of cokernels of maps of vector bundles. The key observation is that one can recover $\QCoh(X)$ from $\Vect(X)$ equipped with the (extra) data of the class of surjective maps.

For Theorem \ref{thm:BLintro}, we first check that pullback induces $D_Z(X) \simeq D_{\pi^{-1}(Z)}(X)$\footnote{Here $D_Z(X) \subset D(X)$ denotes the full subcategory of complexes acyclic on $X-Z$, i.e., the kernel of $D(X) \to D(X-Z)$. In particular, the inclusion $D_Z(X) \subset D(X)$ has a right adjoint $\underline{\Gamma}_Z(-):D(X) \to D_Z(X)$ given by Grothendieck's theory of local cohomology. We prefer working with local cohomology, instead of completions, to access geometry ``near'' $Z$ in Theorem \ref{thm:BLintro} as the former has better compatibility properties with geometric operations on $D(X)$, such as pullback and pushforward.}. The rest of the proof of (1) is then a formalisation of the idea that $D(X)$ is an extension of $D(U)$ by $D_Z(X)$, and that $D(Y)$ is an extension of $D(V)$ by $D_{\pi^{-1}(Z)}(Y)$. Finally, Theorem \ref{thm:TannakaAlgSpaces} immediately yields (2) from (1).

\subsection*{An outline of the paper}

We begin by proving Theorem \ref{thm:TannakaAlgSpaces} in \S \ref{sec:TD}; this section contains the most serious dose of derived algebraic geometry in this paper, and one can find outsider-friendly discussions of some key notions in \cite{GrothIntroInfCat}, \cite{ToenDAGIntro}, and \cite[\S 2--3]{BFNIntegral}. The non-derived analogue of Theorem \ref{thm:TannakaAlgSpaces} for schemes (i.e., Brandenburg and Chirvasitu's Theorem \ref{thm:BC}, as well as Proposition \ref{prop:BCforVect}) is the subject of \S \ref{sec:TDschemes}. Theorem \ref{thm:alglim} is then taken up in \S \ref{sec:jets}; we also discuss examples illustrating the limit of such results. Formal glueing results, including Theorem \ref{thm:BLintro}, are the focus of \S \ref{sec:BL}, though we begin by establishing Corollary \ref{cor:colimitschemescrit} using Theorem \ref{thm:TannakaAlgSpaces}. Theorem \ref{thm:algprod} is proven across \S \ref{sec:algprodsch} and \S \ref{sec:algprodspaces}: the former contains a non-derived proof for schemes using Theorem \ref{thm:BC}, while the latter handles algebraic spaces using Theorem \ref{thm:TannakaAlgSpaces} (and is independent of the former). Finally, the limits of Theorem \ref{thm:algprod} are explored through some examples in \S \ref{sec:algprodex}.

\subsection*{Notation}
We use the language of $\infty$-categories from \cite{LurieHTT}, except that we use the term ``(co)continuous functor'' to describe (co)limit preserving functors. For an $\infty$-category $\calC$, we write $h(\calC)$ for its homotopy-category. For a map $f:K \to L$ in a stable $\infty$-category, we write $\fib(f)$ and $\cofib(f)$ for the fibre and cofibre  respectively. A functor between stable $\infty$-categories is always assumed to be exact.

For a symmetric monoidal $\infty$-category $\calC$, we write $\CAlg(\calC)$ for the $\infty$-category of commutative algebra objects (in the sense of $E_\infty$-rings; see \cite[\S 2.1]{LurieHA}). If $\calC$ is an ordinary category, so is $\CAlg(\calC)$: the latter coincides with the classical defined category of commutative monoids in $\calC$ (by \cite[Example 2.1.3.3]{LurieHA}). We will often use the notion of a dualizable object in a symmetric monoidal $\infty$-category and its basic properties (see \cite[\S 4.2.5]{LurieHA}). In particular, we will freely use that such objects are preserved by symmetric monoidal functors, and that $K \otimes -$ is continuous if $K$ is dualizable.

For symmetric monoidal $\infty$-categories $\calC$ and $\calD$, let $\Fun_{\otimes}(\calC,\calD)$ denote the $\infty$-category of symmetric monoidal functors $\calC \to \calD$. We use a superscript of ``$L$'' to denote the class of cocontinuous functors, while a subscript of ``$c$'' denotes the class of functors preserving compact objects. For example, $\Fun^L_{\otimes,c}(\calC,\calD)$ is the full subcategory of $\Fun_{\otimes}(\calC,\calD)$ spanned by cocontinuous symmetric monoidal functors $F:\calC \to \calD$ that preserve compact objects; these notions are typically used only when $\calC$ and $\calD$ are presentable with enough compact objects. 


For any algebraic space $X$, let $D(X)$ be the quasi-coherent derived category of $X$ (viewed as a symmetric monoidal stable $\infty$-category; see \cite[Definition 2.3.6]{LurieDAGVIII}),  and let $D_\perf(X) \subset D(X)$ be the full subcategory of perfect complexes. For qcqs $X$, we will repeatedly use: $D(X)$ is a compactly generated stable symmetric monoidal $\infty$-category, and $D_\perf(X) \subset D(X)$ coincides simultaneously with the class of compact objects and the class of dualizable objects; see \cite[\S 3.3]{BvdB}, \cite[\S 4]{LipmanNeemanPerfect}, \cite[Tag 09IU]{StacksProject}, \cite[Corollary 2.7.33]{LurieDAGVIII}, and \cite[Corollary 1.5.12]{LurieDAGXII}. We will also use a version of this with supports: if $Z \subset X$ is a constructible closed subspace, then $D_Z(X)$ is compactly generated with compact objects given by  $D_Z(X) \cap D_\perf(X)$, so $D_Z(X) \subset D(X)$ is the smallest full stable $\infty$-subcategory closed under colimits that contains $D_Z(X) \cap D_\perf(X)$; see \cite[Theorem 6.8]{RouqierDimensions} and \cite[Lemma 2.2.1]{SchwedeShipleyStableCat}. All operations involving these objects are assumed to take place in the appropriate $\infty$-categorical sense; for example, if $\{X_i\}$ is a diagram of algebraic spaces, then $\lim D(X_i)$ is the $\infty$-categorical limit in the sense of \cite[\S 4]{LurieHTT}. Let $\QCoh(X) := D(X)^\heartsuit$ be the abelian category of quasi-coherent sheaves, and $D^{cl}(X) := h(D(X))$, i.e., the classical derived category of complexes of $\calO_X$-modules with quasi-coherent cohomology sheaves. Let $\Aff_{/X}$ be the $\infty$-category of affine morphisms over $X$ (in the world of spectral algebraic spaces, so $\Aff_{/X} \simeq \CAlg(D^{\leq 0}(X))$),  and let $\QAff_{/X}$ be $\infty$-category of quasi-affine morphisms over $X$. As a rule, all geometric functors  (such $f^*$, $f_*$, $\Gamma(X,-)$, $\underline{\Gamma}_Z(-)$, etc.) are assumed to be derived, except in \S \ref{sec:TDschemes} and \S \ref{sec:algprodsch}; we will use adornments (such as $\otimes_A^L$ instead of $\otimes_A$) if there is potential for confusion.

In \S \ref{sec:TD}, we will encounter some potentially non-connective commutative algebras $\calA \in \CAlg(D(X))$ for $X$ a qcqs algebraic space. The associated symmetric monoidal $\infty$-category $D(X,\calA)$ of $\calA$-modules in $D(X)$ (denoted $\Mod_{\calA}^{\mathrm{Comm}}(D(X))^{\otimes}$ in \cite[\S 3.3]{LurieHA}) plays an important role, so we give a relatively concrete (but imprecise) description. If $X = \Spec(R)$ is affine, then $D(X) \simeq D(R)$ via $\Gamma(X,-)$, so $\calA$ defines a commutative $R$-algebra  $A := \Gamma(X,\calA)$, and $D(X,\calA) \simeq D(A)$ via $\Gamma(X,-)$ (see \cite[Corollary 3.4.1.9]{LurieHA}), where we write $D(A)$ for the $\infty$-category of $A$-module spectra (denoted $\Mod_A$ in \cite[\S 8]{LurieHA}). In general, if we write $X = \colim U_i$ as a colimit of a diagram $\{U_i\}$ of affine schemes $U_i = \Spec(R_i)$ \'etale over $X$, then $D(X,\calA) \simeq \lim D(U_i,\calA|_{U_i}) \simeq \lim D(A_i)$ via $K \mapsto \{\Gamma(U_i,K|_{U_i})\}$, where $A_i := \Gamma(U_i,\calA)$ is the commutative $R_i$-algebra of global sections of $\calA|_{U_i}$. Moreover, the forgetful functor $D(X,\calA) \to D(X)$ is conservative, continuous  (\cite[Corollary 4.2.3.3]{LurieHA}), and cocontinuous (\cite[Corollary 4.2.3.5]{LurieHA}). In the special case $\calA = \calO_X$, this discussion recovers the \'etale descent equivalence $D(X) \simeq \lim D(U_i) \simeq \lim D(R_i)$. An important non-connective example is given by $\calA = j_* \calO_U$ for $j:U \to X$ a quasi-affine morphism; in this case, $j_*$ induces $D(U) \simeq D(X,\calA)$ (see \cite[Corollary 2.5.16]{LurieDAGVIII}), and the equivalence $D(X,\calA) \simeq \lim D(U_i,\calA|_{U_i})$ above reduces to the \'etale descent equivalence $D(U) \simeq \lim D(U \times_X U_i)$. For further psychological comfort, note that the Raynaud-Gruson d\'evissage (see \cite[Theorem 5.7.6]{RaynaudGruson}, \cite[Tag 08GL]{StacksProject} and \cite[Theorem 1.3.8]{LurieDAGXII}) allows us to choose a {\em finite} diagram $\{U_i\}$ realizing $X$.

\subsection*{Acknowledgements} 
I am very grateful to Vladimir Drinfeld, Johannes Nicaise, and Bjorn Poonen for bringing the algebraization questions treated here to my attention; to Johan de Jong and Ofer Gabber for enlightening discussions, which had a conspicuous influence on this work; to Jacob Lurie, Bertrand To\"{e}n, and Gabriele Vezzosi for conversations and communications that greatly improved my understanding of derived algebraic geometry, and consequently contributed indirectly, but significantly, to \S \ref{sec:TD}; to Daniel Halpern-Leistner and Brandon Levin for useful discussions; and especially to Brian Conrad: his numerous suggestions significantly improved the readability of this paper, and his insistence on the ``correct'' generality  (in the form of comments on an earlier note proving only Theorem \ref{thm:alglim} for schemes using \cite{ToenDerAz} and \cite{BrandenburgChirvasitu}) led to Theorem \ref{thm:TannakaAlgSpaces} in the first place. I was supported by NSF grants DMS 1340424 and DMS 1128155,

\begin{remark}
	On circulation of the main results in this manuscript, we learnt that some of these were known to some experts, at least under mild conditions: Lurie informed us that that Theorem \ref{thm:TannakaAlgSpaces} was familiar to him when $X$ has affine diagonal, and Gabber has told us that he was roughly aware of Theorem \ref{thm:alglim} and Theorem \ref{thm:algprod} through a potential extension of \cite{BrandenburgChirvasitu} to qcqs algebraic spaces.
\end{remark}

\newpage
\section{Tannaka duality for algebraic spaces}
\label{sec:TD}

The goal of this section is to prove the following:

\begin{theorem}
	\label{thm:TD}
	If $X$ and $S$ are qcqs algebraic spaces, then pullback induces isomorphisms 
	\[ \Hom(S,X) \simeq \Fun_{\otimes}(D_\perf(X),D_\perf(S)) \simeq \Fun^L_{\otimes,c}(D(X),D(S)) \simeq \Fun^L_{\otimes}(D(X),D(S)).\]
\end{theorem}

Recall that $\Fun_{\otimes}(D_\perf(X),D_\perf(S))$ is the $\infty$-category of exact symmetric monoidal functors $D_\perf(X) \to D_\perf(S)$. We begin with the purely categorical aspects.

\begin{lemma}
There are natural identifications
\[ \Fun_{\otimes}(D_\perf(X),D_\perf(S)) \simeq \Fun^L_{\otimes,c}(D(X),D(S)) \simeq \Fun^L_{\otimes}(D(X),D(S)).\]
\end{lemma}
\begin{proof}
	The first identification is a consequence of $D(X) = \Ind(D_\perf(X))$. For the second, we must show that every cocontinuous symmetric monoidal functor $F:D(X) \to D(S)$ preserves perfect complexes. As $D_\perf(X) \subset D(X)$ is the full subcategory of dualizable objects by \cite[Corollary 2.7.33]{LurieDAGVIII}, this is an immediate consequence of symmetric monoidal functors preserving dualizable objects.
\end{proof}

The preceding identifications will be used without comment in the sequel. We now check full faithfulness of $\Hom(S,X) \to \Fun_{\otimes}(D_\perf(X),D_\perf(S))$.

\begin{proof}[Proof of full faithfulness]
	The functors $\Hom(-,X)$ and $\Fun_{\otimes}(D_\perf(X),-)$ are stacks for the Zariski (in fact, fpqc), topology, so we may assume $S$ is affine. In this case, any map $S \to X$ is quasi-affine. Thus, the full faithfulness follows from Lurie's theorem \cite[Proposition 3.3.1]{LurieDAGVIII}.
\end{proof}

For essential surjectivity, fix an $F \in \Fun^L_{\otimes,c}(D(X),D(S))$. As before, we are free to localize on $S$, so we assume $S$ is affine. We will use $F$ to progressively build compatible \'etale hypercovers of $X$ and $S$ by (quasi-)affine schemes.  The first, and most essential, step is to ``localize'' algebraic geometry over $S$ in terms of sheaf theory over $X$ via a right adjoint to $F$; if $F$ arises from geometry, then this adjoint is simply the pushforward. The construction of this adjoint highlights the utility of using the functor $F:D(X) \to D(S)$, instead of its restriction to the full subcategories of perfect complexes.

\begin{lemma}
	\label{lem:TDfindadjoint}
$F$ admits a cocontinuous and conservative right adjoint $G:D(S) \to D(X)$. Moreover, $G$ is lax monoidal, and induces a symmetric monoidal equivalence $D(S) \simeq D(X,G \calO_S)$. Under this equivalence, the functor $F:D(X) \to D(S)$ corresponds to the base change functor $D(X) \to D(X,G\calO_S)$.
\end{lemma}

As $G$ is lax monoidal, the object $G\calO_S$ is naturally an $E_\infty$-algebra, i.e., lifts canonically to $\CAlg(D(X))$; this explains the notation $D(X,G\calO_S)$.

\begin{proof}
	The existence of $G$ follows from the adjoint functor theorem as $F$ preserves colimits. The cocontinuity of $G$ is a formal consequence of $F$ preserving compact objects (and $D(X)$ being compactly generated). Moreover, $\Gamma(S,-)$ on $D(S)$ factors through $G$ by adjunction: $\Gamma(S,K) \simeq \Gamma(X,GK)$. As $S$ is affine, it follows that $G$ is conservative. To get the monoidal behaviour, note that the right adjoint of any symmetric monoidal functor is lax monoidal by \cite[Proposition 3.2.1]{LurieDAGVIII}. For the last assertion, we use Barr-Beck-Lurie. To apply this theorem, we must identify the monad resulting from the adjunction as $K \mapsto G\calO_S \otimes K$.   By \cite[Corollary 6.3.5.18]{LurieHA}, it is enough to check that the natural map $G\calO_S \otimes E \to G(F(E))$ is an equivalence for any $E \in D(X)$. In fact, we may restrict to $E \in D_\perf(X)$ by cocontinuity. For such $E$, one checks that $\Hom(K,-)$ applied to either side is $H^0(S, F(K^\vee) \otimes F(E))$ for any $K \in D_\perf(X)$; this proves the claim by Yoneda as $D_\perf(X)$ generates $D(X)$ under colimits. To see that that this equivalence is symmetric monoidal, we must show that the natural map induces an isomorphism $G(K) \otimes_{G\calO_S} G(L) \simeq G(K \otimes L)$ for $K,L \in D(S)$. By cocontinuity, as $D(S)$ is generated under colimits by $\calO_S$, we may assume $K = L = \calO_S$, whence it is clear.
\end{proof}

Using this picture, we can ``pullback'' commutative algebras and modules in $D(X)$ in a tractable way:

\begin{lemma}
	\label{lem:TDCAlg}
	$F$ induces a cocontinuous functor $\CAlg(D(X)) \to \CAlg(D(S))$ with right adjoint $G$. For any $\calA \in \CAlg(D(X))$, there is an induced cocontinuous symmetric monoidal functor $F_\calA:D(X,\calA) \to D(S,F(\calA))$ that preserves compact objects, and is compatible with $F$ under the forgetful functor. For a map $\calA \to \calB$ in $\CAlg(D(X))$, there is a canonical identification $F_{\calB}(L_{\calB/\calA} ) \simeq L_{F(\calB)/F(\calA)}$. 
\end{lemma}
\begin{proof}
	The first assertion comes from \cite[Remark 3.2.2]{LurieDAGVIII}. The equivalence $D(S) \simeq D(X,G\calO_S)$ carries $F(\calA)$ to $\calA \otimes G\calO_S$ by Lemma \ref{lem:TDfindadjoint}. In particular, the desired functor $D(X,\calA) \to D(S,F(\calA))$ is simply the base change functor $D(X,\calA) \to D(X,\calA \otimes G\calO_S)$. One then easily checks that $F_{\calA}$ is cocontinuous, symmetric monoidal, and compatible with $F$; the preservation of compact objects is a consequence of the forgetful right adjoint preserving colimits. Finally, the claim about cotangent complexes is immediate from Lurie's perspective \cite[\S 8.3]{LurieHA} on the functor of points of the cotangent complex in an arbitrary presentable $\infty$-category. More precisely, it follows from the base change formula \cite[Proposition 8.3.3.7]{LurieHA}; see also \cite[Proposition 1.1.2]{LurieDAGXIV} for a similar assertion.
\end{proof}

The next task is to show that $F$ preserves connective objects. For this, we recall a result on quasi-affine maps in the derived setting. First, note that (opposite of) the category $\Aff_{/X}$ of affine $X$-spaces is identified with $\CAlg(D^{\leq 0}(X))$ via pushforward of the structure sheaf. By abuse of notation, for any $Y \in \Aff_{/X}$, we write $\calO_Y \in \CAlg(D^{\leq 0}(X))$  for the corresponding algebra. In the derived setting, this discussion to quasi-affine maps, thanks to a result of Lurie:

\begin{lemma}
	\label{lem:TDQAffDesc}
	Let $f:U \to X$ be a quasi-affine morphism. Then $f_*$ induces a symmetric monoidal equivalence $D(U) \simeq D(X,f_* \calO_U)$. Moreover, the functor $U \mapsto f_* \calO_U$ determines a fully faithful functor $\QAff_{/X}^\opp \to \CAlg(D(X))$.
\end{lemma}
\begin{proof}
	Almost everything can be found in \cite[Proposition 3.2.5 and Lemma 3.2.8]{LurieDAGVIII}. These references do not explicitly state that the equivalence $D(U) \simeq D(X,f_* \calO_U)$ is symmetric monoidal, so we prove it here. Given $K,L \in D(U)$, we must check that $\phi_{K,L}:f_*K \otimes_{f_* \calO_U} f_* L \to f_* (K \otimes_{\calO_U} L)$ is an equivalence. If $K = f^* K'$ for some $K' \in D(X)$, then the claim results from the projection formula. In general, for fixed $L$, the collection of $K \in D(X)$ for which $\phi_{K,L}$ is an equivalence is closed under colimits. As $f$ is quasi-affine, the essential image of the pullback $f^*:D(X) \to D(U)$ generates the target under colimits (as this is true for open immersions and affine maps separately), which implies the claim.
\end{proof}

For any $U \in \QAff_{/X}$, we simply write $\calO_U \in D(X)$ for the pushforward of the structure sheaf. Then the association $U \mapsto \calO_U$ lets us view $\QAff_{/X}^\opp$ as a full subcategory of $\CAlg(D(X))$, and one has a symmetric monoidal identification $D(X,\calO_U) \simeq D(U)$ by Lemma \ref{lem:TDQAffDesc}. Note also that if $U \subset X$ is a quasi-compact open subset, the forgetful functor $D(X,\calO_U) \to D(X)$ lets us view $D(U) \simeq D(X,\calO_U)$ as the right orthogonal of $D_{X\setminus U}(X)$. Using this, we show that $F$ preserves connectivity.

\begin{lemma}
	\label{lem:preserveconnectivity}
	$F$ preserves connective complexes, and thus $G$ preserves coconnective complexes.
\end{lemma}

The possiblity that Lemma \ref{lem:preserveconnectivity} could be true was suggested by an email exchange with Lurie; an earlier version of Theorem \ref{thm:TD} imposed the conclusion of Lemma \ref{lem:preserveconnectivity} as a hypothesis.

\begin{proof}
	By adjunction, it is enough to prove the assertion for $F$. By approximation by perfect complexes (see Lemma \ref{lem:approximateconnective} below), it is enough to check that $F(K) \in D^{\leq 0}_\perf(S)$ if $K \in D_\perf^{\leq 0}(X)$. If not, then there exists a point $s \in S$ such that $F(K)_s \in D_\perf(\kappa(s))$ is non-connective. By replacing $S$ with $\Spec(\kappa(s))$, we may assume $S = \Spec(L)$ for a field $L$. Now $\calA := G\calO_S \in \CAlg(D(X))$ is a field object, i.e., $D(X,\calA) \simeq D(\Spec(L))$ as a symmetric monoidal $\infty$-category (by Lemma \ref{lem:TDfindadjoint}). In particular, $D(X,\calA)$ admits no non-trivial full stable subcategories closed under colimits except itself: such a category would be closed under retracts, so it would contain the unit object, which generates $D(X,\calA) \simeq D(\Spec(L))$ under colimits.  As a special case, if $Z \subset X$ is a constructible closed subset with open complement $U$, then either $D_Z(X,\calA) = D(X,\calA)$ or $\calA \simeq \calA \otimes \calO_U \in D(U)$. We write $[\calA] \in U$ if the latter possibility occurs, and $[\calA] \in X\setminus U$ otherwise. Note that if $[\calA] \in U$, then $F$ factors through $D(X) \to D(U)$ via a functor $D(U) \to D(S)$ that preserves compact objects (as one identifies the latter functor as the base change along $\calO_U \to \calO_U \otimes \calA \simeq \calA$, and then notes that the forgetful right adjoint certainly commutes with direct sums). In the next paragraph, this will be used implicitly in arguments replacing $X$ with $U$.
	
	Choose a sequence $\emptyset = U_0 \subset U_1 \subset \dots U_n = X$ of quasi-compact opens in $X$ such that $U_i$ is the pushout of an \'etale map $V_{i-1} \to U_{i-1}$ along a quasi-compact open immersion $V_{i-1} \hookrightarrow \Spec(A_{i} )$; such presentations always exist (see \cite[Tag 08GL]{StacksProject} or \cite[Theorem 1.3.8]{LurieDAGXII}). Let $Z_i = U_{i} \setminus U_{i-1}$, viewed as a reduced subscheme (say), so $Z_i \simeq \Spec(A_i) \setminus V_{i-1}$ by hypothesis. Choose the minimal $i$ such that $[\calA] \in U_i$. Then $F$ factors through $D(X) \to D(U_i)$, so we may replace $X$ with $U_i$ to assume $i = n$, i.e., that $[\calA] \in Z_{n}$ or, equivalently, that $\calA \in D_{Z_{n}}(X)$. Now $D_{Z_{n}}(X) \simeq D_{Z_{n}}(\Spec(A_{n} ))$ by construction (see Lemma \ref{lem:formalglueingcrit}). Hence, $\calA$ lifts canonically to an object of $D(\Spec(A_{n} ))$; in fact, $\calA \simeq \calA \otimes \calO_{\Spec(A_n)}$. This implies that $F$ factors through the pullback $D(X) \to D(\Spec(A_{n} ))$. Hence, we reduce to the case where $X$ is affine, where everything is clear: any connective perfect complex $K$ is then a retract of a finite colimit of finite free $\calO_X$-modules, so $F(K)$ has the same property on $S$, whence $F(K)$ is connective as $D^{\leq 0}(S)$ contains $\calO_S = F(\calO_X)$ and is closed under retracts and colimits.
\end{proof}

The following lemma was used above.

\begin{lemma}
	\label{lem:approximateconnective}
	Every $K \in D^{\leq 0}(X)$ may be written as a filtered colimit $K = \colim K_i$ with $K_i \in D_\perf^{\leq 0}(X)$.
\end{lemma}

\begin{proof}
	By absolute noetherian approximation (see \cite[Theorem 1.2.2]{CLONagata} or \cite[Tag 07SU]{StacksProject}), we can write $X = \lim X_i$ as cofiltered limit of qcqs and finitely presented $\Z$-spaces $X_i$.  If $f_i:X \to X_i$ is the natural map, then the natural map $ \colim f_i^* f_{i,*} K \to K$ is an isomorphism, so we reduce to the case where $X = X_i$ is noetherian.  As $D(X) = \Ind(D_\perf(X))$, any $K \in D(X)$ can be written as a filtered colimit $K = \colim K_j$ with $K_j \in D_\perf(X)$. If $K$ is connective, then we can also write $K = \colim \tau^{\leq 0} K_j$ (as filtered colimits are exact), so $K$ may be expressed as a filtered colimit of connective coherent complexes (by the noetherian assumption). We may then assume $K$ is itself a bounded coherent connective complex. Fix some $N > 0$. We will construct a diagram 
	\[ K_0 \to K_1 \to K_2 \to K_3 \to \dots \]
	of perfect complexes in $D^{\leq 0}(X)_{/K}$ such that $\cofib(K_i \to K)$ is $(i\cdot N)$-connective. The left-completeness of $D(X)$ (see \cite[Proposition 2.3.18]{LurieDAGVIII}) then gives $\colim K_i \simeq K$, proving the claim. As $K$ is connective, we start with $K_0 = 0$. Fix some $n > 0$, and assume inductively we have constructed a finite tower
	\[ K_0 \to K_1 \to K_2 \to \dots \to K_{n-1} \]
	in $D^{\leq 0}(X)_{/K}$ such that $\cofib(K_i \to K)$ is $(i\cdot N)$-connective for $i \leq n-1$. Let $Q := \cofib(K_{n-1} \to K)$. Choose a connective perfect complex $L$ and a map $L \to Q$ with an $(n\cdot N)$-connective cofibre; this can be done via \cite[Tag 08HH]{StacksProject}. Set $K_n := L \times_Q K$. This gives a map of cofibre sequences
	\[ \xymatrix{ K_{n-1} \ar@{=}[d] \ar[r] & K_n \ar[r] \ar[d] & L \ar[d] \\
			K_{n-1} \ar[r] & K \ar[r] & Q. }\]
	Then $\cofib(K_n \to K)$ is thus $(n \cdot N)$-connective. Continuing in this manner gives the desired diagram.
\end{proof}

\begin{remark}
	As $F$ preserves connectivity, there is an induced adjunction $\adjunction{\QCoh(X)}{H^0F}{\QCoh(S)}{H^0G}$, where $H^0 F$ is the composition 
	\[ \QCoh(X) \hookrightarrow D^{\leq 0}(X) \stackrel{F}{\to} D^{\leq 0}(S) \stackrel{\calH^0}{\to} \QCoh(S),\]
	while $H^0 G$ is the composition
	\[ \QCoh(S) \hookrightarrow D^{\geq 0}(S) \stackrel{G}{\to} D^{\geq 0}(S) \stackrel{\calH^0}{\to} \QCoh(X).\]
	Moreover, the left adjoint $H^0 F$ is symmetric monoidal, while the right adjoint $H^0G$ preserves filtered colimits. It follows formally that $H^0 F:\CAlg(\QCoh(X)) \to \CAlg(\QCoh(S))$ preserves compact objects.
\end{remark}

Recall that we are viewing both $\Aff_{/X}$ and $\QAff_{/X}$ as full subcategories of $\CAlg(D(X))$ via Lemma \ref{lem:TDQAffDesc}. We check that $F$ preserves these subcategories, i.e., one can pullback (quasi-)affine morphisms via $F$:

\begin{lemma}
	\label{lem:TDQAff}
	$F$ induces functors $\Aff_{/X} \to \Aff_{/S}$ and $\QAff_{/X} \to \QAff_{/S}$. For any $U \in \QAff_{/X}$, one has an induced cocontinuous symmetric monoidal functor $F_U:D(X,\calO_U) \to D(S,\calO_{F(U)} )$ that preserves compact objects, is compatible with $F$, and carries $L_{U/X}$ to $L_{F(U)/S}$.
\end{lemma}
\begin{proof}
	The case of affine morphisms is immediate from Lemma \ref{lem:TDCAlg} as  $\Aff_{/X}^\opp \simeq \CAlg(D^{\leq 0}(X))$. Moreover, in this case, $F$ also preserves morphisms of locally almost finite presentation (see Lemma \ref{lem:TDpreservefp}). For the quasi-affine case, fix some quasi-affine map $f:U \to X$. Then we can choose a factorisation $U \stackrel{j}{\hookrightarrow} \overline{U} \stackrel{\pi}{\to} X$ with $\pi$ affine  and $j$ a quasi-compact open immersion. Let $i:Z \hookrightarrow \overline{U}$ be the (constructible) closed complement of $U$, given some finitely presented closed subscheme structure. By the affine case, we obtain an almost finitely presented closed immersion $Z' := F(Z) \hookrightarrow F(\overline{U} ) =: \overline{U'}$ in $\Aff_{/S}$. Let $U' := \overline{U'} \setminus Z'$ be the displayed quasi-compact open subset. Then we claim that $F$ carries $U$ to $U'$, i.e., that $F(\calO_U) \simeq \calO_{U'}$ in $\CAlg(D(X))$. For this assertion, we may replace $X$ with $\overline{U}$ and $S$ with $\overline{U'}$ to assume that $U \subset X$ and $U' \subset S$ are quasi-compact open subsets with constructible closed complements $Z \subset X$ and $Z' \subset S$ respectively. Now note that one has a cofibre sequence
	\[ \underline{\Gamma}_Z(\calO_X) \to \calO_X \to \calO_U \]
	which defines another cofibre sequence
	\[ F(\underline{\Gamma}_Z(\calO_X)) \to \calO_S \to F(\calO_U).\]
We claim that this last sequence coincides with 
\[ \underline{\Gamma}_{Z'}(\calO_S) \to \calO_S \to  \calO_{U'},\]
which certainly implies the desired result. For this, we check that the equivalence $\Phi:D(X,G\calO_S) \simeq D(S)$, given by the inverse of $G$, carries $D_Z(X,G\calO_S)$ onto $D_{Z'}(S)$; one then uses the description of $\underline{\Gamma}_{Z} \to \id_X$ and $\underline{\Gamma}_{Z'} \to \id_S$ as counits of the adjunctions $\adjunction{D_Z(X,G\calO_S)}{ }{D(X,G\calO_S)}{ }$ and $\adjunction{D_{Z'}(S)}{ }{D(S)}{ }$ respectively. The construction of $\Phi$ shows that $\Phi(\calO_Z \otimes G\calO_S) = F(\calO_Z) = \calO_{Z'}$ as commutative algebras. It is thus enough to note that $D_Z(X,G\calO_S)$ is generated under colimits by $(\calO_Z \otimes G\calO_S)$-complexes, and that $D_{Z'}(S)$ is generated under colimits by $\calO_{Z'}$-complexes; for this, one reduces to the affine case by suitable Mayer-Vietoris sequences, and then follows the proof of \cite[Proposition 3.10]{ToenDerAz} or \cite[Lemma 6.17]{LurieDAGXI}. It remains to check that $F_U(L_{U/X} ) \simeq L_{F(U)/S}$. For this, it is enough to check that the identification $D(U) \simeq D(X,\calO_U)$ carries $L_{U/X}$ to $L_{\calO_U/\calO_X}$. If $U \in \Aff_{/X}$, then this is clear. By the transitivity cofibre sequences, we reduce to showing that $L_{\calO_U/\calO_X} = 0$ if $U \subset X$ is a quasi-compact open. Note that $\calO_X \to \calO_U$ is an epimorphism in $\CAlg(D(X))$: one has $\calO_U \otimes \calO_U \simeq \calO_U$ via base change for coherent cohomology (see \cite[Corollary 1.1.3 (3)]{LurieDAGXII}). The  base change formula for cotangent complexes \cite[Proposition 8.3.3.7]{LurieHA} then shows $L_{\calO_U/\calO_X} \simeq 0$.
\end{proof}

The next lemma was used earlier.

\begin{lemma}
	\label{lem:TDpreservefp}
	The functor $F:\Aff_{/X} \to \Aff_{/S}$ preserves morphisms locally of (almost) finite presentation.
\end{lemma}
\begin{proof}
	We first remark that $H^0F:\CAlg(\QCoh(X)) \to \CAlg(\QCoh(S))$ preserves compact objects as $H^0 G$ is compatible with filtered colimits. It follows that if $\calA \in \CAlg(D^{\leq 0}(X))$ is locally of almost finite presentation, then $H^0 F(\calA)$ is finitely presented as an ordinary algebra; here we use that $\calA' \in \CAlg(\QCoh(X))$ is a compact object if and only if the corresponding affine morphism $\underline{\Spec}(\calA') \to X$ is a finitely presented map of classical schemes. To handle higher homotopy groups, we use the characterization of (almost) finite presentation in terms of cotangent complexes in the presence on finite presentation at the classical level (see \cite[Theorem 8.4.3.18]{LurieHA}).
\end{proof}

Recall that a map $g:U \to V$ of qcqs algebraic spaces is \'etale if and only if $L_{U/V} \simeq 0$ and $g$ is locally of almost finite presentation.

\begin{lemma}
	The functor $F:\QAff_{/X} \to \QAff_{/S}$ preserves \'etale morphisms.
\end{lemma}
\begin{proof}
	This is immediate from Lemma \ref{lem:TDQAff}.
\end{proof}

We also have:

\begin{lemma}
	\label{lem:TDetalecover}
	The functor $F:\QAff_{/X} \to \QAff_{/S}$ preserves finite limits and \'etale surjections.
\end{lemma}
\begin{proof}
	The preservation of finite limits follows from the symmetric monoidal assumption on $F$, together with the fact that the fully faithful functor $\QAff_{/X}^\opp \to \CAlg(D(X))$ given by $U \mapsto \calO_U$  preserves finite colimits (which comes from base change for coherent cohomology). Now assume $f:U \to V$ is an \'etale map. Then $f$ is surjective if and only if $D(X,\calO_V) \to D(X,\calO_U)$ is conservative. Thus, it is enough to check that for surjective $f$, the induced functor $D(S,\calO_{F(V)} ) \to D(S,\calO_{F(U)} )$ is conservative. For this, consider the commutative diagram
	\[ \xymatrix{ D(X,\calO_V) \ar[r] \ar[d] & D(X,G\calO_S \otimes \calO_V) \ar[r]^-{\simeq} \ar[d] & D(S, \calO_{F(V)} ) \ar[d]  \\
	D(X,\calO_U) \ar[r] & D(X,G\calO_S \otimes \calO_U) \ar[r]^-{\simeq} &  D(S,\calO_{F(U)} ).}\]
	The second vertical arrow is simply $K \mapsto K \otimes_{\calO_V} \calO_U$, which is conservative by hypothesis. Hence, the last vertical arrow is also conservative, as wanted.
\end{proof}

We can now put the above ingredients together.

\begin{proof}[Proof of Theorem]
	Note first that the theorem is true for $X$ affine (by \cite[Theorem 3.4.2]{LurieDAGVIII} and Lazard's theorem that flat modules are ind-(finite free),  for example). In general, we may choose an \'etale hypercover $\pi^\ast:U^\ast \to X$ with each $U^i$ affine, so $U^i \to X$ is quasi-affine. Then $F(U^\ast) \to S$ is an \'etale hypercover by quasi-affine $S$-schemes by Lemma \ref{lem:TDetalecover}. By the affine case, there is a map $f:F(U^\ast) \to U^\ast$ of simplicial schemes such that the pullback  $f^*:D(X,\calO_{U^i} ) \to D(S,\calO_{F(U^i)} )$ coincides with $F_{U^i}$. Under the \'etale descent identifications $D(X) = \Tot D(X,\calO_{U^\ast} )$ and $D(S) \simeq \Tot D(S,\calO_{F(U^\ast)} )$, one has $\Tot F_{U^i} \simeq F$. It follows that $|f|:|F(U^\ast)| \to |U^\ast|$ is the desired map $S \to X$.
\end{proof}

\begin{remark}
	The previous results give us an identification $\Hom(S,X) \simeq \Fun^L_{\otimes}(D(X),D(S))$, and a fully faithful embedding $\Hom(S,X) \subset \Fun^L_{\otimes}(\QCoh(X),\QCoh(S))$ for qcqs algebraic spaces. We do not know if the latter is an equivalence: it is not clear if every $F \in \Fun^L_{\otimes}(\QCoh(X),\QCoh(S))$ preserves the subcategory of finitely presented quasi-coherent sheaves (= the subcategory of compact objects)\footnote{Gabber has informed us that this obstruction is the only one.}. The identification of compact objects with dualizable objects in $D(X)$ solves this problem in the derived setting. 
\end{remark}

\newpage

\section{The case of schemes, revisited}
\label{sec:TDschemes}

Brandenburg and Chirvasitu \cite{BrandenburgChirvasitu} have shown the following:

\begin{theorem}
	\label{thm:BCbody}
	For qcqs schemes $S$ and $X$, one has $\Hom(S,X) \simeq \Fun^L_{\otimes}(\QCoh(X),\QCoh(S))$.
\end{theorem}

For convenience, we recall the key points of their proof below.

\begin{proof}
	We first prove full faithfulness of $\Hom(S,X) \to \Fun^L_{\otimes}(\QCoh(X),\QCoh(S))$. Say $f,g \in \Hom(S,X)$ admit a symmetric monoidal natural transformation $\eta:f^* \to g^*$; it follows that $\eta$ lifts to a natural transformation of the two induced functors $\CAlg(\QCoh(X)) \to \CAlg(\QCoh(S))$. We will show $f = g$ and $\eta = \id$. Assume first that $S$ and $X$ are affine. Then the map $\eta_{\calO_X}:\calO_S \to \calO_S$ is a ring homomorphism in $\QCoh(S)$, and hence the identity. As $\QCoh(X)$ is generated by $\calO_X$ under colimits, the claim follows in this case. In general, the claim is local on $S$. Moreover, for any closed subset $Z \subset X$, the map $\eta_Z:\calO_{f^{-1}(Z)} \to \calO_{g^{-1}(Z)}$ is a $\calO_S$-algebra map, so $g^{-1}(Z) \subset f^{-1}(Z)$. In particular, we may cover $S$ by affine opens $S_i$ such that both $f|_{S_i}$ and $g|_{S_i}$ factor through some affine open $U_i \subset X$. By replacing $S$ with each element of such a cover, we may assume both $f$ and $g$ factor through an affine open $j:U \hookrightarrow X$.  Both $f^*$ and $g^*$ then factor through $j^*$ as cocontinuous symmetric monoidal functors; here one uses $j^* j_* \simeq \id$. Moreover, one checks that $\eta$ induces a symmetric monoidal natural transformation of the resulting two functors. Thus, by replacing $X$ with $U$, we reduce to the affine case treated earlier.

	For essential surjectivity, fix some functor $F$. As $\Hom(-,X)$ and $\Fun^L_{\otimes}(\QCoh(X),\QCoh(-))$ are fpqc stacks, we may assume $S$ is affine. If $X$ is affine, the claim is clear. In general, for every closed subscheme $Z \subset X$, one has a closed subscheme $f^{-1}(Z) \subset S$ defined via $F(\calO_Z) = \calO_{f^{-1}(Z)}$ with functors $\QCoh(Z) \to \QCoh(f^{-1}(Z))$ and $\QCoh_Z(X) \to \QCoh_{f^{-1}(Z)}(S)$. If $Z$ is constructible with an affine complement $U \subset X$, and $V \subset S \setminus Z'$ is some affine open, one has an induced cocontinuous symmetric monoidal functor $\QCoh(U) \to \QCoh(V)$. As $U$ and $V$ are affine, this arises as pullback along a map $f_{V,U}:V \to X$ factoring through $U$. Using full faithfulness, it is easy to check that the collection $\{f_{V,U}\}$ of maps thus obtained are compatible. It is thus enough to check the collection of all $V$'s obtained by this procedure cover $S$. If not, there exists some $s \in S$ such that $s \in f^{-1}(Z)$ for all $Z \subset X$ closed. Choose an affine open cover $\{U_1,\dots,U_n\}$ of $X$ with complements $Z_i := X \setminus U_i$. Then $\otimes_{i=1}^n \calO_{Z_i} = 0$ as $\cup_i U_i = X$, so $\otimes_{i=1}^n \calO_{f^{-1}(Z_i)} = 0$ as well. On the other hand, $\calO_{f^{-1}(Z_i)} \otimes \kappa(s) \neq 0$, so the tensor product $\otimes_{i=1}^n \calO_{f^{-1}(Z_i)} \neq 0$ as well (since tensor products of non-zero vector spaces are non-zero), which is a contradiction. 
\end{proof}

Recall that a qcqs scheme $X$ is said to have {\em enough vector bundles} if every finitely presented quasi-coherent sheaf can be expressed as the cokernel of a map of vector bundles; any scheme that is quasi-projective over an affine is an example, and the class of such schemes is closed under cofiltered limits with affine transitions. For such schemes, one may go even further than Theorem \ref{thm:BC}

\begin{corollary}
	\label{cor:TDresolution}
	Let $X$ and $S$ be qcqs schemes. Assume $X$ has enough vector bundles. Then 
	\[ \Hom(S,X) \simeq \Fun^L_{\otimes}(\Vect(X),\Vect(S)).\]
\end{corollary}

Here $\Fun^L_{\otimes}(\Vect(X),\Vect(S))$ refers to category of all symmetric monoidal functors $\Vect(X) \to \Vect(S)$ that are right exact; by duality, such functors preserve all exact sequences of vector bundles. The proof below entails building certain functors out of $\QCoh(X)$ starting with functors out of $\Vect(X)$; a more systematic approach is discussed in \S \ref{ss:qcohvect}.

\begin{proof}
	We know $\Hom(S,X) = \Fun_{\otimes}^L(\QCoh(X),\QCoh(S))$, so we will identify the right hand side with $\Fun^L_{\otimes}(\Vect(X),\Vect(S))$. Any symmetric monoidal functor $F:\QCoh(X) \to \QCoh(S)$ preserves vector bundles (as these are the dualizable objects; see Lemma \ref{lem:vectdualizable}), and thus induces a symmetric monoidal functor $\phi(F):\Vect(X) \to \Vect(S)$ that preserves surjections. This construction gives a functor 
	\[ \phi:\Fun^L_{\otimes}(\QCoh(X),\QCoh(S)) \to \Fun^L_{\otimes}(\Vect(X),\Vect(S)).\]
	Next, we claim that any $F \in \Fun^L_{\otimes}(\QCoh(X),\QCoh(S))$ is a left Kan extension of its restriction 
	\[ \psi(F):\Vect(X) \stackrel{\phi(F)}{\to} \Vect(S) \stackrel{i}{\to} \QCoh(S).\]
	This will prove that $\phi$ is fully faithful. To see this, it is enough to note that $\Vect(X) \subset \QCoh(X)$ is a full subcategory that generates $\QCoh(X)$ under colimits (as every finitely presented quasi-coherent sheaf is a cokernel of a map of vector bundles, by assumption). 
	
	It remains to check that $\phi$ is essentially surjective. Given $G \in \Fun^L_{\otimes}(\Vect(X),\Vect(S))$, we will build a cocontinuous symmetric monoidal functor $F:\QCoh(X) \to \QCoh(S)$ extending $G$. For this, we first extend to $\QCoh_{fp}(X)$, so fix some $Q \in \QCoh_{fp}(X)$. Given a ``resolution'' $E_\bullet$ of $Q$, i.e., an exact sequence
	\[ E_2 \to E_1 \to Q \to 1\]
	with $E_i \in \Vect(X)$, we set $F(Q) := \coker(F(E_2) \to F(E_1)) \in \QCoh_{fp}(S)$. We will show that this construction is well-defined (i.e., independent of $E_\bullet$ up to unique isomorphism) and functorial in $Q$. Note first that if $Q \in \Vect(X)$, then $F(Q) = G(Q)$ by the assumption on $G$. To show well-definedness in general, fix a second resolution $G_\bullet$ of $Q$ and a surjective map $\phi_\bullet:E_\bullet \to G_\bullet$ of resolutions; here ``surjective'' simply means that the map $\phi_i:G_i \to E_i$ is surjective for each $i$. Then a diagram chase and the assumption on $G$ show that $\phi_\bullet$ induces an isomorphism
	\[\phi_*:\coker(F(E_2) \to F(E_1)) \simeq \coker(F(G_2) \to F(G_1)).\]
	Note that $\phi_*$ is defined using only $\phi_1$, but the existence of a $\phi_2$ is needed to get a well-defined map. As any two resolutions can be dominated (in the sense of surjections) by a common third one, it follows that $F(Q)$ is well-defined up to isomorphism. 
	
	We next show that $F(Q)$ is well-defined up to unique isomorphism, i.e., the isomorphism $\phi_*$ above is independent of map $\phi_\bullet$ chosen.  Indeed, assume we have two maps $\phi_\bullet,\psi_\bullet:E_\bullet \to G_\bullet$ of resolutions. To show that the induced maps
	\[ \phi_*,\psi_*:\coker(F(E_2) \to F(E_1))  \to \coker(F(G_2) \to F(G_1)) \]
	are the same, we can always replace the resolution $E_\bullet$ by one mapping surjectively onto it (by the argument used to show $F(Q)$ was well-defined up to isomorphism). After doing such a replacement, we can assume that the two maps $\phi_1,\psi_1:E_1 \to G_1$ differ by a map lifting to $G_2$. In this case, the two induced maps
	\[ F(\phi_1),F(\psi_1):F(E_1) \to F(G_1)\]
	differ by a map lifting to $F(G_2)$ by functoriality of $F$ in $\Vect(X)$, and thus the resulting two maps
	\[ \phi_*,\psi_*:\coker(F(E_2) \to F(E_1))  \to \coker(F(G_2) \to F(G_1)) \]
	are visibly the same, which proves that $F(Q)$ is well-defined up to unique isomorphism. 
	
	Next, we make this construction is functorial in $Q$. Given a map $h:Q_1 \to Q_2$, one finds a resolution $E_\bullet$ of $Q_1$, $G_\bullet$ of $Q_2$, and a map $\phi_\bullet:E_\bullet \to G_\bullet$ lifting $h$. This defines a map $\phi_*:F(Q_1) \to F(Q_2)$. Using the trick used to show well-definedness of $F(Q)$ above, one checks that $\phi_*$ is independent of $E_\bullet$, $G_\bullet$, and $\phi_\bullet$. Thus, the construction $Q \mapsto F(Q)$ is functorial in $Q$, so we obtain a functor $F:\QCoh_{fp}(X) \to \QCoh(S)$ which extends $G$, and carries resolutions as above to right exact sequences. As one can lift right exact sequences in $\QCoh_{fp}(X)$ to right exact sequences of resolutions, it follows that $F$ is right exact, so we have produced a finitely cocontinuous functor $F:\QCoh_{fp}(X) \to \QCoh(S)$ extending $G$. By passing to inductive limits, one obtains a cocontinuous functor $F:\QCoh(X) \to \QCoh(S)$ extending $G$. We leave it to the reader to check that one may endow $F$ with the structure of a symmetric monoidal functor extending the given one on $G$ in a unique (and evident) way, which is enough to prove the desired claim.
\end{proof}

\begin{lemma}
	\label{lem:vectdualizable}
Let $X$ be a qcqs scheme. Then $E \in \QCoh(X)$ is dualizable if and only if $E$ is a vector bundle.
\end{lemma}
\begin{proof}
	It is clear that vector bundles are dualizable. Conversely, assume $E \in \QCoh(X)$ is dualizable with dual $E^\vee$. To show $E$ is a vector bundle, by localising, we may assume $X = \Spec(A)$ is affine. We now identify $\QCoh(X)$ with $\Mod_A$ to solve the corresponding question for modules. Then $\Hom(E,-) = E^\vee \otimes (-)$ commutes with filtered colimits, and thus $E$ is finitely presented. Similarly, $\Hom(E^\vee,-) = E \otimes (-)$, so $E$ is flat. Any finitely presented flat $A$-module is finite locally free, proving the claim.
\end{proof}

One may wonder if the cocontinuity condition on the functors appearing on the right hand side of Corollary \ref{cor:TDresolution} is automatically satisfied: the next two examples show this is not the case, and that such functors abound in nature. Moreover, these examples also indicate a potential subtlety in the applying  Corollary \ref{cor:TDresolution}: the condition that a map in $\Vect(S)$ be surjective is defined in terms of the ambient category $\QCoh(S)$, and is not intrinsic to the category $\Vect(S)$. One may raise similar objections to Theorem \ref{thm:TD}, but they are easily refuted: it is almost impossible (certainly quite unnatural) to write down a non-exact functor $D_\perf(X) \to D_\perf(S)$, and the exactness condition is intrinsic to the $\infty$-categories of perfect complexes.

\begin{example}
	\label{ex:surjqproj}
	Let $X$ be an affine regular noetherian scheme of dimension $2$, and let $x \in X$ be a closed point. Set $S = X \setminus \{x\}$. Then the inclusion $j:S \to X$ induces an equivalence $j^*:\Vect(X) \to \Vect(S)$ by the Auslander-Buschbaum formula; explicitly, for any $E \in \Vect(S)$, the double dual $\overline{E}^{\ast \ast}$ of any $\overline{E} \in \Coh(X)$ extending $E$ is a vector bundle extending $E$. However, the map $j$ is certainly not an isomorphism. This does not contradict Corollary \ref{cor:TDresolution} as the symmetric monoidal equivalence $\Vect(S) \simeq \Vect(X)$  does {\em not} preserve surjections: the inverse to $j^*$ is given by reflexivising a coherent extension, and the reflexivisation process loses surjectivity properties at a missing point. Explicitly, if $X = \A^2 = \Spec(k[y,z])$ over a field $k$ and $x = (0,0)$, then the map $\calO_X^{\oplus 2} \stackrel{(y,z)}{\to} \calO_X$ in $\Vect(X)$ is surjective over $S$, but not at $x$.
\end{example}

Example \ref{ex:surjqproj} might lead one to suspect that such phenomenon can be avoided in the projective case. However, this is not the case:

\begin{example}
	\label{ex:surjproj}
	Fix a field $k$, and let $X = \P^1$ and $S = \Spec(k)$. We will construct a symmetric monoidal functor $F:\Vect(\P^1) \to \Vect(S)$ which does not come from $k$-point of $\P^1$. Our functor $F$ will not preserve surjections. To construct $F$, consider the natural map $\pi:\A^2 \setminus \{0\} \to \P^1$. Then $\pi^*:\Vect(\P^1) \to \Vect(\A^2-\{0\} )$ is certainly symmetric monoidal. As in Example \ref{ex:surjqproj}, we know $j^*:\Vect(\A^2) \to \Vect(\A^2 \setminus \{0\} )$ is an equivalence, where $j:\A^2 \setminus \{0\} \to \A^2$ is the defining map. Thus, we find a symmetric monoidal functor $F:\Vect(\P^1) \to \Vect(k)$ given via $F = i^* \circ (j^*)^{-1} \circ \pi^*$, where $i:S \to \A^2$ is the origin. One can check easily that $F$ does not come from geometry. In fact, the surjection $\calO_X^{\oplus 2} \to \calO_X(1)$ (defining $\P^1$) is carried by $F$ to the $0$ map $\calO_S^{\oplus 2} \to F(\calO_X(1)) \simeq \calO_S$. 
\end{example}

\subsection{Recovering $\QCoh$ from $\Vect$}
\label{ss:qcohvect}

Fix a qcqs scheme $X$ with enough vector bundles. Examples \ref{ex:surjqproj} and \ref{ex:surjproj} show that there is no way to recover $\QCoh(X)$ from $\Vect(X)$ as there is no hope knowing what ``surjective'' maps should be intrinsically in terms of $\Vect(X)$. We sketch now why this is the only obstruction: one functorially recovers $\QCoh(X)$ from $\Vect(X)$ equipped with (the extra data of) the class of ``surjective maps.'' The ideas below already appear in the proof of Corollary \ref{cor:TDresolution} implicitly, and will not be used elsewhere in the paper. We begin by defining the ambient category where all constructions will take place.

\begin{definition}
	Let $\widehat{\Vect}(X)$ be the category of additive presheaves $\Vect(X)^\opp \to \Ab$ on $\Vect(X)$.
\end{definition}

The map $\Vect(X) \to \widehat{\Vect}(X)$ enjoys a good universal property.

\begin{lemma}
	\label{lem:TDVectUnivProp}
	The category $\widehat{\Vect}(X)$ is a cocomplete abelian category, and the Yoneda embedding $\Vect(X) \to \widehat{\Vect}(X)$ is the universal additive map $\Vect(X) \to \calA$ to a cocomplete abelian category. 
\end{lemma}
\begin{proof}
	Left to the reader (see \cite[Proposition 3.6]{KellyColimits} as well as \cite{DayLackLimits}).
\end{proof}

Roughly speaking, we view the extra data of the class of surjective maps in $\Vect(X)$ as a topology on $\Vect(X)$; the category $\widehat{\Vect}(X)$ is then the category of presheaves, while the category of interest will be the category for sheaves. However, to avoid discussing topologies on additive categories, we encode the data of surjections in terms of their coequalizers to isolate the objects of interest.

\begin{definition}
	Let $S$ be the class of maps in $\widehat{\Vect}(X)$ of the form $\coker\Big(E \times_Q E \stackrel{p_1-p_2}{\to} E\Big) \to Q$, where $E \to Q$ is a surjective map in $\Vect(X)$. Let $\calC \subset \widehat{\Vect}(X)$ be the full subcategory of {\em $S$-local presheaves}, i.e., presheaves $F$ such that $1 \to F(Q) \to F(E) \stackrel{p_1-p_2}{\to} F(E \times_Q E)$ is exact for every surjection $E \to Q$ or, equivalently, that $F(Q) \simeq F(T)$ for any map $T \to Q$ in $S$.
\end{definition}

The Yoneda embedding $\Vect(X) \to \widehat{\Vect}(X)$ lands in $\calC$, and the basic properties of $\calC$, summarized in the next lemma, are close analogues of the usual properties of sheaves on a coherent site.

\begin{lemma}
$\calC$ is closed under limits and filtered colimits in $\widehat{\Vect}(X)$. The inclusion $i:\calC \to \widehat{\Vect}(X)$ admits an exact left adjoint $L:\widehat{\Vect}(X) \to \calC$. In particular, $\calC$ admits all limits and colimits.
\end{lemma}

\begin{proof}
We first define a functor $L^+:\widehat{\Vect}(X) \to \widehat{\Vect}(X)$ via the familiar formula from sheaf theory, i.e., 
\[ L^+(F)(Q) = \colim \ker(F(E) \to F(E \times_Q E)),\]
where the colimit is indexed by the (opposite of the) category $I(Q)$ of surjective maps $E \to Q$. Using the class $S$ introduced above, we can rewrite this as
\[ L^+F(Q) = \colim F(T),\]
where the colimit is indexed by the (opposite of the) category $J(Q)$ of maps $T \to Q$ in $S$. Note that each map in $S$ is a monomorphism, so $J(Q)$ is a poset, and a subposet of the poset $\mathrm{Sub}(Q)$ of subobjects of $Q$ in $\widehat{\Vect}(X)$. As a fibre product of bundle surjections defines a square of bundle surjections, $J(Q)$ is stable under finite intersections in $\mathrm{Sub}(Q)$, so the second colimit defining $L^+(F)(Q)$ is filtered. Let us temporarily call $F \in \widehat{\Vect}(X)$ {\em separated} if $F(Q) \to F(E)$ is injective for any $E \to Q$ in $I(Q)$. Using the second description, one checks that $L^+(F)$ is always separated for any $F \in \widehat{\Vect}(X)$.   Moreover, if $F$ is separated, then a diagram chase shows that $L^+(F) \in \calC$. In particular, the functor $L^+ \circ L^+$ can be viewed as a functor $L:\widehat{\Vect}(X) \to \calC$.  It is also clear that $L(F) = L^+(F) = F$ if $F \in \calC$, and one then checks that $L$ gives a left adjoint to $i$ (using that any map $F \to G$ with $G \in \calC$ extends canonically to a map $L^+(F) \to G$). Finally, the second description of $L^+$ given above shows that $L^+$ is exact, and hence so is $L$.
\end{proof}

As in sheaf theory, the left-adjoint $L:\widehat{\Vect}(X) \to \calC$ is a localization of $\widehat{\Vect}(X)$.

\begin{lemma}
	\label{lem:TDVectUnivPropVect}
	Any cocontinuous functor $G:\widehat{\Vect}(X) \to \calA$ to a cocomplete abelian category $\calA$ which carries maps in $S$ to isomorphisms factors uniquely over $L$, i.e., $T \simeq T \circ i \circ L$.
\end{lemma}
\begin{proof}
	For notational convenience, let $M = i \circ L$. It is enough to show that for any $F \in \widehat{\Vect}(X)$, the map $F \to M(F)$ is in the smallest strongly saturated class $\overline{S}$ of maps in $\widehat{\Vect}(X)$ containing $S$; here a subcategory $\overline{S} \subset \mathrm{Mor}([1],\widehat{\Vect}(X))$ is called a strongly saturated class if it is closed under colimits, satisfies the $2$-out-of-$3$ property, and if maps in $\overline{S}$ are stable under arbitrary pushouts in $\widehat{\Vect}(X)$; see \cite[Definition 5.5.4.5]{LurieHTT} for more. Indeed, for any $G:\widehat{\Vect}(X) \to \calA$ as in the lemma, the collection of maps $f$ in $\widehat{\Vect}(X)$ such that $G(f)$ is an isomorphism is a strongly saturated class containing $S$, and thus $G$ carries $\overline{S}$ to isomorphisms, so $G(F) \simeq G(M(F))$ as wanted. To finish, it is enough to note $F \to M(F)$ is in $\overline{S}$ by general facts about Bousfield localizations; see \cite[Proposition 5.5.4.15]{LurieHTT}.
\end{proof}

This leads to a simple universal property describing the Yoneda embedding $\Vect(X) \to \calC$.

\begin{corollary}
	\label{lem:TDVectUnivPropC}
	Any exact additive functor $F:\Vect(X) \to \calA$ to a cocomplete additive category $\calA$ extends uniquely to a cocontinuous functor $\calC \to \calA$.
\end{corollary}
\begin{proof}
By Lemma \ref{lem:TDVectUnivProp}, any such $F$ extends uniquely to a cocontinuous functor $\widehat{F}:\widehat{\Vect}(X) \to \calA$. We claim that $\widehat{F}$ carries maps in $S$ to isomorphisms. To see this, fix a surjection $E \to Q$ of bundles. The kernel $K$ is a bundle as well, so we obtain a short exact sequence
\[ 1 \to K \to E \to Q \to 1\]
in $\Vect(X)$. As $F$ is exact, the sequence
\[1 \to F(K) \to F(E) \to F(Q) \to 1\]
is also exact. Using the identification $E \times_Q E \simeq E \times K$, it follows that $F(E \times_Q E) \simeq F(E) \times_{F(Q)} F(E)$. This implies that $F$ carries the map $\coker(E \times_Q E \stackrel{p_1-p_2}{\to} E) \to Q$ in $S$ to an isomorphism. Lemma \ref{lem:TDVectUnivPropVect} then shows that $\widehat{F}$ factors through a cocontinuous functor $\calC \to \calA$, proving the claim. 
\end{proof}

The category $\calC$ constructed above is actually a familiar object.

	\begin{proposition}
		The functor $\QCoh(X) \to \widehat{\Vect}(X)$ given by $F \mapsto \Hom(-,F)$ factors through an equivalence $\QCoh(X) \simeq \calC$. 
	\end{proposition}
	\begin{proof}
		The association $F \mapsto \Hom(-,F)$ gives a continuous functor $\QCoh(X) \to \widehat{\Vect}(X)$ that factors through $\calC$, thus giving a continuous functor $\Phi:\QCoh(X) \to \calC$. We first check that $\Phi$ is cocontinuous. For this, it is enough to check that $\Phi$ preserves filtered colimits and finite colimits separately. For finite colimits, we must check $\Phi$ is right exact. If $K \to F \to Q \to 1$ is a right exact sequence in $\QCoh(X)$, then the only non-trivial claim is that $\Phi(F) \to \Phi(Q)$ is surjective in $\calC$. By definition of $\calC$, we must check that for any map $E_1 \to Q$ with $E_1 \in \Vect(X)$, there exists a surjection $E_2 \to E_1$ such that the composite $E_2 \to Q$ lifts to $F$; this follows from the existence of enough vector bundles on $X$, the fact that vector bundles are compact in $\QCoh(X)$, and approximation by finitely presented quasi-coherent sheaves (see Lemma \ref{lem:DeligneIndQC}). To show preservation under filtered colimits, as $i:\calC \to \widehat{\Vect}(X)$ preserves filtered colimits, it is enough to check that the composite $\QCoh(X) \to \widehat{\Vect}(X)$ preserves filtered colimits; this is a consequence of objects in $\Vect(X)$ being compact in both $\QCoh(X)$ and $\widehat{\Vect}(X)$.
		
		To show $\Phi$ is an equivalence, we construct an inverse $\Psi:\calC \to \QCoh(X)$. The canonical inclusion $\Vect(X) \to \QCoh(X)$ is exact. By the universal property in Lemma \ref{lem:TDVectUnivPropC}, this inclusion factors through a cocontinuous functor $\Psi:\calC \to \QCoh(X)$. As both $\Phi$ and $\Psi$ are cocontinuous, and because both $\QCoh(X)$ and $\calC$ are generated under colimits by $\Vect(X)$, it is enough to check that $\Psi \circ \Phi$ and $\Phi \circ \Psi$ are the identity on $\Vect(X)$, which is obvious.
	\end{proof}

This gives a characterization of $\QCoh(X)$ in terms of $\Vect(X)$ equipped with the notion of surjections.

	\begin{corollary}
Any exact additive functor $\Vect(X) \to \calA$ to a cocomplete abelian category $\calA$ factors uniquely through a cocontinuous functor $\QCoh(X) \to \calA$.
	\end{corollary}


\newpage
\section{Algebraization of jets}
\label{sec:jets}

The goal of this section is to prove the following theorem.

\begin{theorem}
	\label{thm:alglimitlqs}
Let $A$ be a ring equipped with an ideal $I$ such that $A$ is $I$-adically complete. Fix a qcqs algebraic space $X$. Then $X(A) \simeq \lim X(A/I^n)$ via the natural map.
\end{theorem}

\begin{proof}
	We first prove injectivity, though this step could be avoided by using the argument in the next paragraph together with Proposition \ref{prop:colimitschemes} (which is independent of this \S). Fix two maps $f,g:\Spec(A) \to X$ that induce the same map $\epsilon_n:\Spec(A/I^n) \to X$.  Let $Z \to \Spec(A)$ be the pullback of $\Delta:X \to X \times X$ along $(f,g):\Spec(A) \to X \times X$. Then $Z \to \Spec(A)$ is a quasi-compact monomorphism with a specified system of compatible sections over each $\Spec(A/I^n)$. In particular, $Z$ is quasi-affine. By Example \ref{ex:alglimitQA} (or simply the next paragraph), one obtains a map $\Spec(A) \to Z$ such that the composite $\Spec(A) \to Z \to \Spec(A)$ agrees with the identity modulo $I^n$ for all $n$. In other words, $Z \to \Spec(A)$ is a quasi-compact monomorphism with a section, and hence an isomorphism, which proves $f=g$.

	For surjectivity, fix a compatible system of maps $\{\epsilon_n:\Spec(A/I^n) \to X\}$. Pullback gives a compatible system $\{\epsilon_n^* : D_\perf(X) \to D_\perf(A/I^n)\}$ of exact symmetric monoidal functors, and thus an exact symmetric monoidal functor
	\[ F:D_\perf(X) \to \lim D_\perf(A/I^n).\]
	By Lemma \ref{lem:perfcomplexnilpext} for the special case $I_n := I^n$, the canonical map $D_\perf(A) \to \lim D_\perf(A/I^n)$ is an equivalence, so $F$ may be viewed as an exact symmetric monoidal functor $D_\perf(X) \to D_\perf(A)$. By Theorem \ref{thm:TD}, this comes from a unique map $\epsilon:\Spec(A) \to X$ with $F \simeq \epsilon^\ast$. It is clear by construction that $\epsilon$ extends each $\epsilon_n$ (as $F$ extends $\epsilon_n^*$), which gives the surjectivity of $X(A) \to \lim X(A/I^n)$. 
\end{proof}

The following patching result for perfect complexes is a crucial ingredient in the proof above.

\begin{lemma}
	\label{lem:perfcomplexnilpext}
	Let $A$ be a ring equipped with a descending sequence $\{\dots I_n \subset  \dots \subset I_2 \subset I_1\}$ of ideals. Assume that $A \simeq \lim A/I_n$, and that $I_n/I_{n+1}$ is nilpotent for each $n$. Then $D_\perf(A) \simeq \lim D_\perf(A/I_n)$.
\end{lemma}
\begin{proof}
	The full faithfulness of $D_\perf(A) \to \lim D_\perf(A/I_n)$ is easy to check using that $K \otimes_A - $ commutes with limits if $K \in D_\perf(A)$. The essential surjectivity is standard; see \cite[Tag 09AW]{StacksProject} or \cite[Corollary 12]{FaltingsVeryRamified}. The key point is represent an object $\{K_n\} \in \lim D_\perf(A/I_n)$ by a projective system $\{P_n\}$ of complexes of finite projective $A/I_n$-modules such that the transition map $P_{n+1} \to P_n$ induces an isomorphism $P_{n+1}/I_{n} \simeq P_{n}$ of chain complexes (not merely in the derived category);  this can be done by induction on $n$ using the nilpotence assumption and \cite[Tag 09AR]{StacksProject}.
\end{proof}

\begin{remark}
	\label{rmk:alglimadm}
	As pointed out by Nicaise, Theorem \ref{thm:alglimitlqs} and its proof apply to any {\em admissible} topological ring $A$, i.e., a ring $A$ equipped with ideals $\{I_n\}$ as in Lemma \ref{lem:perfcomplexnilpext}. An elementary example of such a ring, which is not adic as in Theorem \ref{thm:alglimitlqs}, is:  $A = k\llbracket x,y \rrbracket$ with $I_k = (x \cdot y^k)$. An example that is more relevant in $p$-adic geometry comes from crystalline cohomology: $A := \widehat{\F_p \langle x \rangle}$, the completed divided power polynomial algebra on a generator $x$, with $I_k = \overline{\langle x \rangle^{[k]}} = \overline{(\gamma_n(x))_{n \geq k}}$ being the $k^{\mathrm{th}}$-level of the Hodge filtration. 
\end{remark}

Theorem \ref{thm:alglimitlqs} immediately implies the following result identifying jet spaces and the Greenberg functor at infinite level as honest mapping functors. 

\begin{corollary} 
	\label{cor:arcGreenberg}
	Let $X$ be a qcqs space over some base ring $A$.
	\begin{enumerate} 
		\item The arc space functor $R \mapsto \lim X(R[t]/(t^n))$ is isomorphic to the functor $R \mapsto X(R\llbracket t \rrbracket)$ on the category of $A$-algebras.
		\item If $A = k$ is a perfect field $k$ of characteristic $p > 0$, the infinite level Greenberg functor $R \mapsto \lim X(W_n(R))$ is isomorphic to the functor $R \mapsto X(W(R))$ on the category of $k$-algebras.
	\end{enumerate}
\end{corollary}
\begin{proof}
	(1) is immediate from Theorem \ref{thm:alglimitlqs}, while (2) follows from Remark \ref{rmk:alglimadm} as the kernel of $W_{n+1}(R) \to W_n(R)$ is nilpotent.
\end{proof}

\begin{remark}
	\label{rmk:chow}
	Theorem \ref{thm:alglimitlqs} implies (and is equivalent to) an existence result for sections. Let $A$ be an $I$-adically complete ring for some ideal $I$, set $S = \Spec(A)$, and fix a qcqs map $\nu:Y \to S$ of algebraic spaces.  If we write $\widehat{X}$ for the formal $I$-adic completion of an $S$-space $X$, then Theorem \ref{thm:alglimitlqs} is equivalent to
	\[ \Hom_S(S,Y) \simeq \Hom_{\widehat{S}}(\widehat{S},\widehat{Y}).\]
In particular, every formal section of $\nu$ is automatically algebraic. Even when $A$ is noetherian, such results typically impose stronger hypotheses on the diagonal: \cite[Theorem 5.4.1]{EGA3_1} requires $\pi$ to be separated, while \cite[Theorem 5.4.1]{LurieDAGXII} requires $\pi$ to have affine diagonal. By Example \ref{ex:propernoetheriancolimit}, a similar remark also applies to $\Hom_S(X,Y)$, where $X$ is a proper $A$-space with $A$ noetherian.  Moreover, using the trick in Remark \ref{rmk:Gabberlimit}, if one works exclusively with schemes, one can drop all conditions on the diagonal: one has 
\[ \Hom_S(X,Y) \simeq \Hom_{\widehat{S}}(\widehat{X},\widehat{Y}) \]
for a proper $S$-scheme $X$ and an arbitrary $S$-scheme $Y$, i.e., \cite[Theorem 5.4.1]{EGA3_1} is true for all $Y$.
\end{remark}

\begin{remark}[Gabber]
	\label{rmk:Gabberlimit}
	Theorem \ref{thm:alglimitlqs} applies to any scheme $X$. Indeed, to prove $X(A) = \lim X(A/I^n)$, we may assume $X$ is quasi-compact. Any quasi-compact scheme is a {\em reasonable} algebraic space in the sense of \cite[Tag 03I8]{StacksProject}, i.e., there exists a surjective \'etale map $U \to X$ with bounded fibre degree\footnote{A typical unreasonable example is $\A^1_{\C}/\Z$, where $\Z$ acts by translation.}, where $U$ is an affine scheme. By inspection of the proof of \cite[Tag 03K0]{StacksProject}, we can write $X = \colim X_i$ where the colimit is filtered and computed in the category of fpqc sheaves, $X_i \to X$ is a local isomorphism (in particular, it is \'etale), and $X_i$ is qcqs. Given a compatible system $\{\epsilon_n:\Spec(A/I^n) \to X\}$, there exists some $i$ such that $\epsilon_1$ lifts to a map $\widetilde{\epsilon_1}:\Spec(A/I) \to X_i$. By the infinitesimal lifting property, each map $\epsilon_n$ factors through a unique map $\widetilde{\epsilon_n}:\Spec(A/I^n) \to X_i$ lifting $\widetilde{\epsilon_1}$. By the qcqs case, one finds a map $\widetilde{\epsilon}:\Spec(A) \to X_i$ extending each $\widetilde{\epsilon_n}$. Composing back down to $X$ then shows that $X(A) \to \lim X(A/I^n)$ surjective. For injectivity, one repeats the trick involving the diagonal in the proof of Theorem \ref{thm:alglimitlqs}, using the surjectivity just proven in lieu of Example \ref{ex:alglimitQA}. More generally, the same argument applies to any Zariski locally reasonable algebraic space. Similarly, Corollary \ref{cor:arcGreenberg} also extends to all schemes.
\end{remark}

\begin{remark}
	Theorem \ref{thm:alglimitlqs} and Remark \ref{rmk:Gabberlimit} show: if $A$ is a ring which is $I$-adically complete for some ideal $I$, then $\colim \Spec(A/I^n) \simeq \Spec(A)$ in the category of schemes. This gives an instance of the general phenomenon that colimits in schemes can be quite different from the corresponding colimit in fpqc sheaves. Indeed, the colimit of the diagram $\{\Spec(A/I^n)\}$ as an fpqc sheaf over $\Spec(A)$, which may be viewed as an ind-scheme, has no $A$-points unless $I$ is nilpotent. In particular, this shows that (the perfect complex component of) Theorem \ref{thm:TD} does not extend to ind-schemes.
\end{remark}

We give a few examples where Theorem  \ref{thm:alglimitlqs} can be proved directly by classical methods.

\begin{example}
	\label{ex:alglimitQA}
	Assume $X$ is quasi-affine. Then the affine case immediately shows that $X(A) \to \lim X(A/I^n)$ is injective. For surjectivity, fix maps $\epsilon_n:\Spec(A/I^n) \to X$ compatible in $X$. Choose an affine $Y$ containing $X$ as a quasi-compact open with constructible closed complement $Z$. Then the compsite maps $\mu_n:\Spec(A/I^n) \to Y$ extends to a unique map $\mu:\Spec(A) \to Y$. For $K \in D_{Z,\perf}(Y)$, one has 
	\[ \mu^* K = \lim \mu_n^* K = 0,\]
	where the first equality uses that $\mu^* K$ is perfect, so $- \otimes_{\calO_{\Spec(A)}} \mu^* K$ commutes with limits, while the second equality is a simple consequence of $\mu_n$ factoring through $X$. Writing $\calO_Z$ as a colimit of such $K$ then implies $\mu^* \calO_Z = 0$, i.e., $\mu^{-1}(Z) = \emptyset$, which implies that $\mu$ factors through a map $\epsilon:\Spec(A) \to X$, as desired.
\end{example}

\begin{example}
	\label{ex:alglimqproj}
	Assume $X = \P^m$. A compatible system $\{\epsilon_n:\Spec(A/I^n) \to X\}$ of maps is determined by a compatible system of objects $\{L_n,s_{n,0},\dots,s_{n,m}\}$ comprising an invertible $(A/I^n)$-module $L_n$ and elements $s_{n,i} \in L_n$ such that the induced map $(A/I^n)^{\oplus m} \stackrel{s_{n,i}}{\to} L_n$ is surjective.  Then $L := \lim L_n$ is an invertible $A$-module by Lemma \ref{lem:vectcont}, and $s_i := \lim s_{n,i}$ is an element of $L$ such that $A^{\oplus m} \stackrel{s_i}{\to} L$ is surjective (by Nakayama). This gives a map $\epsilon:\Spec(A) \to X$ lifting $\epsilon_n$ for each $n$. One checks that $\epsilon$ is the unique such extension, which proves $X(A) \simeq \lim X(A/I^n)$ in this case. More generally, the same argument applies to any scheme that is quasi-projective over an affine.
\end{example}

\begin{example}
	Assume $A$ is noetherian, and $X$ is separated. Then we can write $X = \lim X_i$ with $X_i$ finitely presented separated $\Z$-schemes, and all transition maps $X_i \to X_j$ affine. As the assertion $X(A) = \lim X(A/I^n)$ is compatible with inverse limits in $X$, we may reduce to the case $X = X_i$ is a finitely presented $\Z$-scheme. Now $\lim X(A/I^n)$ is exactly the set of sections of the map $\widehat{X \times \Spec(A)} \to \mathrm{Spf}(A)$ obtained as the formal $I$-adic completion of the map $X \times \Spec(A) \to \Spec(A)$. By formal GAGA, which applies as both $A$ and $X$ are noetherian, each such section is (uniquely) algebraizable by \cite[Tag 0899]{StacksProject}, which immediately gives $X(A) \simeq \lim X(A/I^n)$.
\end{example}

The following lemma was used Example \ref{ex:alglimqproj} and mentioned in the introduction.

\begin{lemma}
	\label{lem:vectcont}
Let $A$ be a ring, $I \subset A$ an ideal, and assume $A \simeq \lim A/I^n$. Then $\Vect(A) \simeq \lim \Vect(A/I^n)$.
\end{lemma}
\begin{proof}
	This follows from Lemma \ref{lem:perfcomplexnilpext} using \cite[Corollary 2.7.33]{LurieDAGVIII} to identify $\Vect(A)$ as the dualizable objects of $D_\perf(A)$. Alternatively, we can argue directly as follows. The natural functor $F:\Vect(A) \to \lim \Vect(A/I^n)$ is fully faithful on finite free modules by assumption. By passage to summands, it follows that $F$ is fully faithful. For essential surjectivity, we must check: if $\{P_n\} \in \lim \Vect(A/I^n)$, then $P := \lim P_n$ is a finite projective $A$-module. Fix a projector $\epsilon_1 \in M_N(A/I)$ cutting out $P_1$, i.e., we fix a surjection $f_1:(A/I)^{\oplus N} \to P_1$ as well as a section $s_1:P_1 \to (A/I)^{\oplus N}$ such that $s_1 \circ f_1 = \epsilon_1$. Lifting sections gives a compatible system $\{ f_n:(A/I^n)^{\oplus N} \to P_n \}$ of surjections. By using the projectivity of $P_n$ over $A/I^n$, one inductively constructs a compatible system of sections $\{s_n:P_n \to (A/I^n)^{\oplus N}\}$ of $\{f_n\}$. The composition $\{\epsilon_n := s_n \circ f_n\}$ is a projector in $M_N(A) = \lim M_N(A/I^n)$ that defines $P$, so $P$ is projective.
\end{proof}

We also give an example showing that Theorem \ref{thm:alglimitlqs} does not extend to all algebraic stacks. 

\begin{example}
	\label{ex:abvar}
	Let $A/\C$ be an abelian variety, and consider the classifying stack $X = B(A)$. Then $X(R)$ is the groupoid of $A$-torsors over $R$, for any ring $R$. We will construct a geometrically regular $\C$-algebra $R$ which is complete along an ideal $I$, and a compatible system of $A$-torsors $Y_n \to \Spec(R/I^n)$ such that the order of each $Y_n$ in $H^1(\Spec(R/I^n),A)$ is infinite. Given such data, if there exists an $A$-torsor $Y \to \Spec(R)$ lifting each $Y_n$, then the order of $Y$ would have to be $\infty$, which cannot happen: the group $H^1(\Spec(R),A)$ is torsion as $R$ is regular by \cite[Proposition XIII.2.6]{RaynaudAmpleLB}. It follows that the family $\{Y_n\}$ defines a point of $\lim X(R/I^n)$  that does not lift to $X(R)$. To construct this data, take $R$ to be the completion of $\A^2$ along an irreducible nodal cubic $C \subset \A^2$ with $\Spec(R/I) = C$. Note that $\Spec(R)$ is geometrically regular over $\A^2$, and hence certainly so over $\Spec(\C)$. The enlarged fundamental pro-group of $\Spec(R/I)$ admits $\Z$ as a quotient. Choosing a non-torsion point $P \in A(\C)$ then gives an $A$-torsor $Y_1 \to \Spec(R/I)$ with infinite order; explicitly, we glue the trivial $A$-torsor over the normalisation of $C$ to itself using translation by $P$ as the isomorphism over the node. The map $H^1(\Spec(R/I^{n+1} ),A) \to H^1(\Spec(R/I^n),A)$ is surjective (in fact, bijective) as the obstruction to lifting an $A$-torsor $Y_n \to \Spec(R/I^n)$, viewed as a map $f_n:\Spec(R/I^n) \to B(A)$, to an $A$-torsor $Y_{n+1} \to \Spec(R/I^{n+1} )$ is an element of 
	\[ \Hom(f_n^* L_{B(A)}, I^n/I^{n+1}[1]) \simeq H^2(\Spec(R/I), \omega_{Y_1}^\vee \otimes I^n/I^{n+1} ),\]
	which is $0$. Thus, we find the desired family $\{Y_n\}$ by inductively lifting.
\end{example}

The construction in Example \ref{ex:abvar} can be modified to show that Theorem \ref{thm:alglimitlqs} fails for Artin $2$-stacks.

\begin{example}
	\label{ex:2stacklimit}	
Set $X = K(\G_m,2)$. Then $\pi_0 (X(S)) = H^2(\Spec(S),\G_m)$ for any ring $S$. We will construct a geometrically regular $\C$-algebra $R$ that is complete along an ideal $I$ such that $H^2(\Spec(R/I^n),\G_m) \simeq H^2(\Spec(R/I),\G_m)$ contains a point of infinite order. As $R$ is regular, one knows $H^2(\Spec(R),\G_m)$ is torsion by a standard result in \'etale cohomology (see \cite[Corollary IV.2.6]{MilneEC}), so the map $X(R) \to \lim X(R/I^n)$ is not essentially surjective. For the construction, take  $R_0$ to be the glueing two copies of $\A^2$ along a curve $C \subset \A^2$ of geometric genus $\geq 1$ over $\C$. Then $\Pic(C) \subset H^2(\Spec(R_0),\G_m)$ via a suitable Mayer-Vietoris sequence, so $H^2(\Spec(R_0),\G_m)$ certainly contains points of infinite order. Set $R$ to be completion of a suitable surjection $P \to R_0$ with $P$ a polynomial ring, and $I = \ker(R \to R_0)$. Then this pair $(R,I)$ satisfies the desired properties: $R$ is geometrically regular by construction, and one has $H^2(\Spec(R/I^n),\G_m) = H^2(\Spec(R/I),\G_m) = H^2(\Spec(R_0),\G_m)$ by deformation theory.
\end{example}

\newpage

\section{Formal glueing}
\label{sec:BL}

The goal of this section is to revisit the Beauville-Laszlo theorem \cite{BLDescent} and variants, and prove Theorem \ref{thm:BLintro}.  We begin by noting that the proof of Theorem \ref{thm:alglimitlqs} actually gives a criterion for deciding when a qcqs algebraic space is a colimit of a diagram of such spaces.

\begin{proposition}
	\label{prop:colimitschemes}
Let $X$ be a qcqs algebraic space, and let $\{X_i\}$ be a diagram of qcqs $X$-spaces. Assume that $D(X) \simeq \lim D(X_i)$ via the natural maps. Then $D_\perf(X) \simeq \lim D_\perf(X_i)$, and this implies that $X = \colim X_i$ in the category of qcqs algebraic spaces.
\end{proposition}
\begin{proof}
	For the first part, we use that $D_\perf(S)$ is exactly the category of dualizable objects in $D(S)$ for any qcqs algebraic space $S$ by \cite[Lemma 8.16]{LurieDAGXI}. The claim follows as duals are computed ``pointwise'' in limit of symmetric monoidal $\infty$-categories by \cite[Proposition 4.2.5.11]{LurieHA}. The second part now follows from Theorem \ref{thm:TD} as the natural map $D_\perf(X) \to \lim D_\perf(X_i)$ is exact and symmetric monoidal.
\end{proof}

\begin{remark}
	\label{rmk:colimitbasechange}
	Let $X$ and $\{X_i\}$ be as in Proposition \ref{prop:colimitschemes}. If $D(X) \simeq \lim D(X_i)$, then $X \simeq \colim X_i$ as above. However, one can say a bit more: these colimits are {\em universal} with respect to flat qcqs maps $g:Y \to X$, i.e., one has $Y \simeq \colim Y \times_X X_i$ via the natural map. To prove this, it suffices to show that $D(Y) \simeq \lim D(Y \times_X X_i)$ via the natural map. By Lurie's version \cite[Corollary 8.22]{LurieDAGXI} of \cite{BFNIntegral} and flatness, we know $D(Y \times_X S) \simeq D(Y) \otimes_{D(X)} D(S)$ for any qcqs $X$-space $S$; we note here that a qcqs algebraic space is perfect in the sense of \cite[Definition 8.14]{LurieDAGXI} by \cite[Corollary 1.5.12]{LurieDAGXII}.  Using this base change formula, a formal argument involving the diagonal $\Delta:Y \to Y \times_X Y$ shows that $D(Y)$ is self-dual as a $D(X)$-module: the map $\mathrm{coev}:D(X) \to D(Y \times_X Y)$ is $\Delta_* \circ g^*$, while $\mathrm{ev}:D(Y \times_X Y) \to D(X)$ is $g_* \circ \Delta^*$ (see \cite[Corollary 4.7]{BFNIntegral}). In particular, the functor $- \otimes_{D(X)} D(Y)$ is continuous, so
	\[ D(Y) \simeq D(Y) \otimes_{D(X)} \lim D(X_i) \simeq \lim \Big(D(Y) \otimes_{D(X)} D(X_i)\Big) \simeq \lim D(Y \times_X X_i),\]
	which proves the stability of these colimits under flat base change. 
\end{remark}

\begin{remark}
The implication $\Big(D_\perf(X) \simeq \lim D_\perf(X_i)\Big) \Rightarrow \Big(X \simeq \colim X_i\Big)$ in Proposition \ref{prop:colimitschemes} cannot be reversed. For example, consider a non-trivial finite group $G$ and let $X_\bullet$ be the usual simplicial scheme presenting $BG$ over a field $k$. Then the colimit of $X_\bullet$ as an algebraic space is simply $\Spec(k)$. On the other hand, $\lim D_\perf(X_\bullet) \simeq D_\perf(BG)$ is not $D_\perf(\Spec(k))$.
\end{remark}

\begin{remark}
	The implication $\Big(D(X) \simeq \lim D(X_i)\Big) \Rightarrow \Big(D_\perf(X) \simeq \lim D_\perf(X_i)\Big)$ in Proposition \ref{prop:colimitschemes} cannot be reversed. For example, consider $X = \Spec(\Z_p)$ and $X_i := \Spec(\Z/p^n)$. Then $D_\perf(X) \simeq \lim D_\perf(X_i)$ as in Theorem \ref{thm:alglimitlqs}. On the other hand, $D(X) \neq \lim D(X_i)$: the $\Z_p$-module $\Q_p$ maps to $0$ in each $D(X_i)$, and hence also in $\lim D(X_i)$, so $D(X) \to \lim D(X_i)$ is not even faithful.
\end{remark}

As a special case of Proposition \ref{prop:colimitschemes}, one obtains the following description of formal projective space.

\begin{example}
	\label{ex:propernoetheriancolimit}
	Let $A$ be a noetherian ring that is $I$-adically complete for an ideal $I$. Let $X$ be a proper $A$-space. Then $X = \colim X_n$ in the category of qcqs spaces, where $X_n = X \times_{\Spec(A)} \Spec(A/I^n)$ is the (classical) displayed fibre product. Indeed, by Proposition \ref{prop:colimitschemes}, it is enough to show $D_\perf(X) \simeq \lim D_\perf(X_n)$, which is a consequence of formal GAGA. (For example, one can use \cite[Theorem 5.3.2]{LurieDAGXII} and the observation that a pseudo-coherent complex $K \in D(X)$ is perfect if and only if $K|_{X_1}$ is so.) Moreover, by Remark \ref{rmk:Gabberlimit}, if $X$ is a scheme, then $X \simeq \colim X_n$ in the category of all schemes.
\end{example}

We can now explain how to generalize certain well-known ``formal glueing'' results, such as \cite{BLDescent}. We begin with the following criterion for establishing such glueing features in a general setting. 

\begin{proposition}
\label{prop:formalglueing}
	Fix a qcqs algebraic space $X$ equipped with a constructible closed subspace $Z \subset X$. Let $\pi:Y \to X$ be a qcqs map algebraic spaces such that $\pi^*$ induces an equivalence $D_Z(X) \simeq D_{\pi^{-1}(Z)}(Y)$. Let $U = X \setminus Z$, and $V = Y \setminus \pi^{-1}(Z)$. Then one has:
	\begin{enumerate}[(1)]
		\item The natural map $\Phi:D(X) \to D(Y) \times_{D(V)} D(U)$ is an equivalence.
		\item $\Phi$ induces $D_\perf(X) \simeq D_\perf(Y) \times_{D_\perf(V)} D_\perf(U)$.
		\item $\Phi$ induces $\Vect(X) \simeq \Vect(Y) \times_{\Vect(V)} \Vect(U)$.
		\item If $\pi$ is flat, then $\Phi$ induces $\QCoh(X) \simeq \QCoh(Y) \times_{\QCoh(V)} \QCoh(U)$.
		\item The fibre square
			\[ \xymatrix{ V \ar[r]^j \ar[d]^\pi & Y \ar[d]^\pi \\
				U \ar[r]^j & X }\]
		 	is a pushout in qcqs algebraic spaces. 
		\item All of the above are also true after flat qcqs base change on $X$.
	\end{enumerate}
\end{proposition}

\begin{proof}
	Write objects in $D(Y) \times_{D(V)} D(U)$ as triples $(K,L,\eta)$ where $K \in D(Y)$, $L \in D(U)$, and $\eta:j^* K \simeq \pi^* L$. 	We first check (1). For full faithfulness, we may work locally on $X$.  As all pushforward and pullback functors involved are cocontinuous, by the projection formula, we reduce to checking that 
\begin{equation}
\label{eq:formalglueingstructuresheaf}
\calO_X \stackrel{a}{\to} \pi_* \calO_Y \oplus j_* \calO_U \stackrel{b}{\to} \pi_* j_* \calO_V 
\end{equation}
	is a cofibre sequence. This can be checked after applying the conservative functor $\underline{\Gamma}_Z(-) \oplus \Big(- \otimes j_* \calO_U\Big)$. The latter follows from the following sequence of assertions: the map $\underline{\Gamma}_Z(a)$ is an isomorphism, the map $\underline{\Gamma}_Z(b)$ is $0$, and the map $a \otimes j_* \calO_U$ is the inclusion of a direct summand with the projection on the complement given by $b \otimes j_* \calO_U$. These result from base change in coherent cohomology, and the assumption $\pi^*:D_Z(X) \simeq D_{\pi^{-1}(Z)}(Y)$. For essential surjectivity, there is a functor $\Psi:D(Y) \times_{D(V)} D(U) \to D(X)$ given as follows: if $g := (K,L,\eta) \in D(Y) \times_{D(V)} D(U)$ as before, then set 
	\[ \Psi(g) := \fib\Big(\pi_* K \oplus j_* L \to (\pi \circ j)_* j^* K\Big),\]
	where the map is induced by $\eta$. Note by construction that $\Psi$ is the right adjoint\footnote{The existence of the right adjoint can also be seen as follows: \cite[Proposition 5.5.3.13]{LurieHTT} and \cite[Theorem 1.1.4.4]{LurieHA} show that the fibre product $D(Y) \times_{D(V)} D(U)$ is a presentable stable $\infty$-category, \cite[Proposition 5.5.3.12]{LurieHTT} shows that $\Phi$ preserves colimits, whence \cite[Corollary 5.5.2.9]{LurieHTT} gives the right adjoint.} to $\Phi$. We will check that $\Psi \circ \Phi \simeq \id$ and $\Phi \circ \Psi \simeq \id$ via the (co)units of the adjunction. The first assertion is automatic from the full faithfulness of $\Phi$. For the second, note that $j^* \Psi(g) = L$, while $\underline{\Gamma}_Z(\Psi(g)) \simeq \underline{\Gamma}_Z(\pi_* K)$ corresponds to $\underline{\Gamma}_{\pi^{-1}(Z)}(K)$ under the equivalence $D_{\pi^{-1}(Z)}(Y) \simeq D_Z(X)$ induced by $\pi$. Now for $(K,L,\eta)$ as above, one has a canonical cofibre sequence
	\[ \underline{\Gamma}_{\pi^{-1}(Z)}(K) \to K \to j_* L\]
	coming from $\eta$. One then checks $\Phi(\Psi(g)) \to g$ induces isomorphisms after projection to $D(Y)$ (as it induces isomorphisms after applying $\underline{\Gamma}_{\pi^{-1}(Z)}(-)$ and $j^*$) and $D(U)$ separately, and thus $\Phi \circ \Psi \simeq \id$.
	
	To get (2), we repeat the first half of the proof of Proposition \ref{prop:colimitschemes}. 
	
	For (3), we argue as in (2) using that the essential image of the fully faithful symmetric monoidal functor $\Vect(S) \to D(S)$ is exactly the set of dualizable objects in $D_\perf^{\leq 0}(S)$ for any qcqs space $S$ by \cite[Corollary 2.7.33]{LurieDAGVIII}. In order to apply this, we must check that $K \in D_\perf(X)$ is connective if and only if $\pi^* K$ and $j^* K$ are connective. The forward direction is clear. For the converse, as connectivity of perfect complexes can be detected are restriction to a stratification, it is enough to check that $\pi:Y \to X$ is surjective over $Z$; if this was false, then $\pi^*:D_Z(X) \to D_{\pi^{-1}(Z)}(Y)$ would have a non-trivial kernel (given by the structure sheaf of the residue field at the missing point), which contradicts the assumption. 
	
	For (4), the flatness of $\pi$ shows $\Phi$ restricts to a functor $\QCoh(X) \to \QCoh(Y) \times_{\QCoh(V)} \QCoh(U)$. Thanks to (1), it is now enough to check: for $K \in D(X)$, if $j^* K$ and $\pi^* K$ are quasi-coherent sheaves, so is $K$. For such $K$, we know $\calH^i(K) \in D_Z(X)$ for $i \neq 0$ as $j^* K$ is a sheaf. As $\pi$ is flat and $\pi^* K$ is a sheaf, it follows that $\pi^* \calH^i(K) = \calH^i(\pi^* K) = 0$ for $i \neq 0$, and thus $\calH^i(K) = 0$ for $i \neq 0$  by the assumption $\pi^*:D_Z(X) \simeq D_{\pi^{-1}(Z)}(X)$, so $K \simeq \calH^0(K)$ is a sheaf.
	
	Lastly, (5) follows from Proposition \ref{prop:colimitschemes} and (2) (or (1)), while (6) comes from Remark \ref{rmk:colimitbasechange}.
\end{proof}

\begin{remark}
	In Proposition \ref{prop:formalglueing}, it is important to work with $\infty$-categories instead of their $1$-categorical truncations as the formation of the homotopy-category is incompatible with fibre products. More concretely, the natural map $D^{cl}(X) \to D^{cl}(Y) \times_{D^{cl}(V)} D^{cl}(U)$ is essentially surjective and full (which follows from the $\infty$-categorical assertion for formal reasons), but can often fail to be faithful. Intuitively, this happens as $D^{cl}(Y) \times_{D^{cl}(V)} D^{cl}(U)$ forgets ``how'' objects over $Y$ and $U$ are being identified over $V$. An explicit example illustrating this failure is given in Example \ref{ex:BLabelianfailure}.
\end{remark}

Proposition \ref{prop:formalglueing} specializes to a few commonly encountered geometric situations. To illustrate these, let $\pi:Y \to X$ be a qcqs map of qcqs algebraic spaces, and fix a constructible closed subspace $Z \subset X$. The hypothesis $D_Z(X) \simeq D_{\pi^{-1}(Z)}(Y)$ from Proposition \ref{prop:formalglueing} is satisfied (and consequently the conclusions there apply) notably in the following examples:
		
\begin{example}
	\label{ex:formalglueflat}
	The map $\pi$ is flat, and an isomorphism over $Z$. For example, $\pi$ could be an \'etale neighbourhood of $Z$ in $X$, or one could take $X$ to be a noetherian affine with $Y$ the completion of $X$ along $Z$ (in the sense of ring theory). The assumption $D_Z(X) \simeq D_{\pi^{-1}(Z)}(Y)$ for such $\pi$ comes from Lemma \ref{lem:formalglueingcrit} (2) below. The consequence $D(X) \simeq D(Y) \times_{D(V)} D(U)$ may be viewed as a derived variant of formal glueing results due to Artin \cite[\S 2]{ArtinAFMII} (which pertains to noetherian case) and Ferrand-Raynaud \cite{FerrandRaynaudDescent} (which is literally Proposition \ref{prop:formalglueing} (4)); see also \cite[Tag 05ER]{StacksProject} for more references. 
\end{example}

\begin{example}	
	\label{ex:BLcurve}
	The space $X$ is a separated smooth scheme of relative dimension $d$ over some base ring $R$, $Z \subset X$ is the image of a section $s:\Spec(R) \to X$, and $Y = \Spec(\lim A_n)$, where $A_n = \Gamma(Z_n,\calO_{Z_n})$ is the ring of functions on the $n$-fold infinitesimal thickening on $Z$ in $X$. The assumption $D_Z(X) \simeq D_{\pi^{-1}(Z)}(Y)$ comes from Lemma \ref{lem:formalglueingcrit} (2) (or (1)) below. Zariski locally on $\Spec(R)$, one may choose a local co-ordinates defining $Z \subset X$, so $A := \lim A_n \simeq R\llbracket t_1,\dots,t_d\rrbracket$ and $V =  \Spec(A) \setminus \Spec(A/(t_1,\dots,t_d))$. The consequence $\Vect(X) \simeq \Vect(Y) \times_{\Vect(V)} \Vect(U)$, in the special case where $d = 1$,  recovers the Beauville-Laszlo theorem \cite[\S 4, Example]{BLDescent} in the form in which it is often used. Note that this is not covered by Example \ref{ex:formalglueflat} as completions often fail to be flat in the non-noetherian case. In fact, if $R$ is not a coherent ring, then $A \simeq R\llbracket t \rrbracket$ can fail to be even $R$-flat; see \cite{ChaseCoherent}.
\end{example}

\begin{example}
	\label{ex:BL}
The space $X = \Spec(A)$ is an affine scheme for some ring $A$, the closed subspace $Z \subset X$ is cut out by a regular element $t \in A$, and $Y = \Spec(\widehat{A})$, where $\widehat{A} = \lim A/t^n$. The assumption $D_Z(X) \simeq D_{\pi^{-1}(Z)}(Y)$ comes from Lemma \ref{lem:formalglueingcrit} (2) (or (1)) below. This case recovers the Beauville-Laszlo equivalence \cite[\S 3, Theorem]{BLDescent} 
			\[ \Mod_t(A) \simeq \Mod_t(\widehat{A}) \times_{\Mod(\widehat{A}[\frac{1}{t}])} \Mod(A[\frac{1}{t}]),\] 
			as explained in Corollary \ref{cor:BLexactly}; here $\Mod_t(-) \subset \Mod(-)$ is the full subcategory of all $t$-regular modules. Note further that this equivalence does not extend to all $A$-modules; see Example \ref{ex:BLabelianfailure}. Proposition \ref{prop:formalglueing} shows that such an equivalence can be salvaged at the derived level, i.e., the failure of the classical statement for modules is the cost of ignoring $\Tor$ groups.
\end{example}

\begin{remark}
It is commonly asserted that the glueing result discussed in Example \ref{ex:formalglueflat} is a direct consequence of faithfully flat descent for the covering $g:Y \sqcup U \to X$. However, this is not clear to us: the latter statement would realize $\Vect(X)$ as the equalizer of the two evident maps $\Vect(Y) \times \Vect(U) \to \Vect(Y \times_X Y) \times \Vect(V) \times \Vect(U)$, which entails recording an isomorphism on $Y \times_X Y$ (and higher fibre products, if $\Vect(-)$ is replaced by $D(-)$) as part of the descent data. It is nevertheless a {\em consequence} of the formal glueing result that this additional data is extraneous.
\end{remark}

The next lemma was used above.
	
\begin{lemma}
	\label{lem:formalglueingcrit}
		Let $\pi:Y \to X$ be a map of qcqs algebraic spaces. Fix a finitely presented closed subspace $Z \subset X$. Assume one the following:
		\begin{enumerate}
			\item $\pi$ is quasi-affine, and there exists a $K \in D_\perf(X)$ with support $Z$ such that $K \simeq \pi_* \pi^* K$.
			\item $\pi$ is an isomorphism over $Z$ in the derived sense, i.e., $Z \times^L_X Y \simeq Z$. 
		\end{enumerate}
		Then $\pi^*$ induces $D_Z(X) \simeq D_{\pi^{-1}(Z)}(Y)$.
	\end{lemma}

	\begin{proof}
		First consider (1). The claim is local, so we may assume $X$ is affine, and $Y$ is quasi-affine. In this case, by Thomason's \cite[Lemma 3.14]{ThomasonClassification} and approximation by perfect complexes with support constraints, the smallest stable subcategory $\langle K \rangle$ of $D(X)$ containing $K$ and closed under colimits is exactly $D_Z(X)$, and similarly $\langle L \rangle = D_{\pi^{-1}(Z)}(Y)$ for $L = \pi^* K$; here we use that any stable subcategory of $D(X)$ closed under colimits is automatically an ideal (as $\langle \calO_X \rangle = D(X)$ and $\langle \calO_Y \rangle = D(Y)$, since $X$ and $Y$ are quasi-affine). As both $\pi^*$ and $\pi_*$ are cocontinuous, it follows that $\pi_* \circ \pi^* \simeq \id$ on $D_Z(X)$, and $\pi^* \circ \pi_* \simeq \id$ on $D_{\pi^{-1}(Z)}(Y)$ as the same is true on generators by base change in coherent cohomology.

		Now consider (2). We first check that $\pi_* \pi^* K \simeq K$ for $K \in D_Z(X)$. By cocontinuity, one may assume $K$ is compact. By filtering $K$ suitably, one reduces to the case where $K$ comes from an $\calO_Z$-complex (but is not necessarily compact in $D(X)$ any more).  The claim now follows by base change in coherent cohomology and the assumption on $Z$. It remains to check that $L \simeq \pi^* \pi_* L$ for $L \in D_{\pi^{-1}(Z)}(Y)$. If $L$ comes from an $\calO_{\pi^{-1}(Z)}$-complex, then this follows from base change in coherent cohomology. One then reduces to this case by cocontinuity and compact generation of $D_{\pi^{-1}(Z)}(Y)$.
	\end{proof}

\begin{remark}
	Lemma \ref{lem:formalglueingcrit} (2) is closely related to \cite[Theorem 2.6.3]{ThomasonTrobaugh}; the latter imposes a stronger flatness constraint. Note also that the hypothesis of finite presentation on $Z$ is necessary. Indeed, otherwise we may take $X = \Spec(\colim \C \llbracket t^{\frac{1}{n}} \rrbracket)$, and $Y = Z = \Spec(\C)$ with the map $t^{\frac{1}{n}} \mapsto 0$ for all $n$.
\end{remark}

We also explain why Proposition \ref{prop:formalglueing} recovers the classical Beauville-Laszlo theorem.

\begin{corollary}
	\label{cor:BLexactly}
Fix notation as in Example \ref{ex:BL}. The base change functor gives an equivalence
\[ \phi:\Mod_t(A) \simeq \Mod_t(\widehat{A}) \times_{\Mod(\widehat{A}[\frac{1}{t}])} \Mod(A[\frac{1}{t}]).\]
\end{corollary}
\begin{proof}
The map $\phi$ has a right adjoint $\psi$ given by $(M,N,\eta) \mapsto \ker(M \oplus N \to M[\frac{1}{t}])$ with evident notation. We first check that $\psi \circ \phi \simeq \id$. For this, fix some $M \in \Mod_t(A)$. Tensoring the resulting exact sequence 
	\[ 0 \to M \to M[\frac{1}{t}] \to M[\frac{1}{t}]/M \to 0\]
	with $\widehat{A}$, as $- \otimes^L_A \widehat{A}$ is the identity on $t^\infty$-torsion modules, one concludes that $M \to M[\frac{1}{t}]$ induces an isomorphism on $\pi_i(- \otimes_A^L \widehat{A})$ for $i > 0$. The proof of Proposition \ref{prop:formalglueing} (1) then shows that
	\[ 0 \to M \to \Big(M \otimes_A \widehat{A}\Big) \oplus M[\frac{1}{t}] \to M \otimes_A \widehat{A}[\frac{1}{t}] \to 0\]
	is an exact sequence, so $\psi \circ \phi \simeq \id$, and thus $\phi$ is fully faithful. For essential surjectivity, fix some $M \in \Mod_t(\widehat{A})$, $N \in \Mod(A[\frac{1}{t}])$, and an isomorphism $\eta:M[\frac{1}{t}] \simeq N \otimes_A \widehat{A}$. Define $K := M \times_{M[\frac{1}{t}]} \Big(N \otimes^L_{A} \widehat{A}\Big)$ as a fibre product in $D(\widehat{A})$. Then projection induces $\mu:K[\frac{1}{t}] \simeq N \otimes^L_{A} \widehat{A}$, and thus an identification $\fib(K \to M) \simeq \fib(N \otimes^L_{A} \widehat{A} \to N \otimes_{A} \widehat{A})$. In particular, this fibre is connected and uniquely $t$-divisible. The triple $(K,N,\mu)$ defines an object of $D(\widehat{A}) \times_{D(\widehat{A}[\frac{1}{t}])} D(A[\frac{1}{t}])$, and thus comes from a unique $L \in D(A)$ by Proposition \ref{prop:formalglueing} (1). As $L[\frac{1}{t}] \simeq N$, we know that $H^i(L)$ is $t^\infty$-torsion for $i \neq 0$. We also have
	\[ L \otimes_A^L A/(t) \simeq L \otimes_A^L \widehat{A} \otimes_{\widehat{A}}^L A/(t) \simeq K \otimes^L_{\widehat{A}} A/(t) \simeq M/(t),\]
	where the last equality uses that $M$ is $t$-regular, and that $K \to M$ had a uniquely $t$-divisible fibre.  It follows that $H^i(L)$ is uniquely $t$-divisible for $i \neq 0,1$, $H^0(L)$ is $t$-regular, and $H^1(L)$ is $t$-divisible. The previous reduction then shows $H^i(L) \simeq 0$ for $i \neq 0,1$. As $- \otimes_A^L \widehat{A}$ is the identity operation on $t^\infty$-torsion modules, we get $H^1(L) \simeq H^1(L) \otimes_A^L \widehat{A} \simeq H^1(L \otimes_A^L \widehat{A}) = H^1(K) = 0$. Thus, $L \simeq H^0(L)$ is a $t$-regular module, and one checks that $\phi(L) = (M,N,\eta)$, proving the claim.
\end{proof}

The proof of Proposition \ref{prop:formalglueing} takes place in the derived category. This is necessary: the equivalence $\Phi$ in Proposition \ref{prop:formalglueing} does not induce an equivalence on the {\em abelian} categories of quasi-coherent sheaves. An example illustrating this failure (coming from $C^\infty$-function theory) is mentioned in \cite[\S 4, Remark 4]{BLDescent}, so we recall a different one (coming from rigid analytic geometry) here for the convenience of the reader. 

\begin{example}
	\label{ex:BLabelianfailure}
	Fix a prime $p$, and let $A$ be the ring of germs of bounded algebraic functions at $0$ on the $p$-adic unit disc, i.e., 
	\[ A = \colim \Big(\Z_p[x] \to \Z_p[\frac{x}{p}] \to \Z_p[\frac{x}{p^2}] \to \dots\Big).\] 
	Note that both $p$ and $x$ are regular elements of $A$. In fact, $A$ is a domain: we may view $A$ as the subring of $\Q_p[x]$ spanned by polynomials $f(x)$ with $f(0) \in \Z_p$. Thus $x$ is uniquely $p$-divisible in $A$ by construction, so $A/p^n \simeq \Z/p^n$, and thus $\widehat{A} = \Z_p$ (where the completion is $p$-adic). Set $X = \Spec(A)$, $Y = \Spec(\widehat{A})$, $Z = \Spec(A/p)$ with $U$ and $V$ as in Proposition \ref{prop:formalglueing}. Now if we consider $M = A/(x)$, then  the map 
\[ \eta:M \to \big(M \otimes_A \widehat{A}\big) \oplus M[\frac{1}{p}]\]
has a non-trivial kernel $K$: the elements $0 \neq \frac{x}{p^n} \in M$ for all $n \geq 1$ span a copy of $\Q_p/\Z_p$ in $K$. As $M \otimes_A \widehat{A} \simeq \widehat{A}/(x) \simeq \Z_p$ and $M[\frac{1}{p}]$ are both $p$-torsion free, studying $\Hom(A,M)$ then shows that 
\[ \QCoh(X) \to \QCoh(Y) \times_{\QCoh(V)} \QCoh(U) \]
	is not faithful. This failure is explained by the derived picture as follows: the sequence
	\[ M \to \Big(M \otimes^L_A \widehat{A}\Big) \oplus M[\frac{1}{p}] \to M \otimes_A^L \widehat{A}[\frac{1}{p}] \]
	is a cofibre sequence in $D(X) \simeq D(A)$ by Proposition \ref{prop:formalglueing} (1), but the sequence of $A$-modules 
	\[ 0 \to M \stackrel{\eta}{\to}\Big(M \otimes_A \widehat{A}\Big) \oplus M[\frac{1}{p}] \to M \otimes_A \widehat{A}[\frac{1}{p}] \]
	obtained by applying $\pi_0(-)$ to the above cofibre sequence is not exact on the left. In fact, one computes 
	\[ M \otimes_A^L \widehat{A} \simeq \cofib(A \stackrel{x}{\to} A) \otimes_A^L \widehat{A} \simeq \cofib(\widehat{A} \stackrel{x}{\to} \widehat{A}) \simeq \widehat{A} \oplus \widehat{A}[1] \simeq \Z_p \oplus \Z_p[1],\]
	where the second-to-last equality uses that $x = 0$ on $\widehat{A}$. This gives 
	\[ M \otimes^L_A \widehat{A}[\frac{1}{p}] \simeq \widehat{A}[\frac{1}{p}] \oplus \widehat{A}[\frac{1}{p}][1] \simeq \Q_p \oplus \Q_p[1]\]
	by inverting $p$, and thus the kernel $K = \ker(\eta)$ above is identified as 
	\[ K \simeq \coker\Big(\pi_1(M \otimes_A^L \widehat{A}) \to \pi_1(M \otimes_A^L \widehat{A}[\frac{1}{p}])\Big) \simeq \widehat{A}[\frac{1}{p}]/\widehat{A} \simeq \Q_p/\Z_p.\]
	Note further that this example also shows that Proposition \ref{prop:formalglueing} is {\em not} true for classical derived categories. Indeed, write $\Phi^{cl}:D^{cl}(X) \to D^{cl}(Y) \times_{D^{cl}(V)} D^{cl}(U)$ for the obvious functor. Then one computes
	\[ \Hom(\calO_X,\widetilde{M}) \simeq M \]
	in $D^{cl}(X)$, while
	\[ \Hom(\Phi^{cl}(\calO_X),\Phi^{cl}(\widetilde{M})) \simeq M/K\]
	in $D^{cl}(Y) \times_{D^{cl}(V)} D^{cl}(U)$. Of course, the latter is entirely a consequence of the non-faithfulness of 
	\[ h\Big(D(Y) \times_{D(V)} D(U)\Big) \to D^{cl}(Y) \times_{D^{cl}(V)} D^{cl}(U).\]
\end{example}

\newpage
\section{Algebraization of products: schemes}
\label{sec:algprodsch}

The goal of this section is to prove the following theorem.

\begin{theorem}
	\label{thm:algprodsch}
	Fix a set $I$ of rings $\{A_i\}_{i \in I}$ with product $A := \prod_i A_i$, and a qcqs scheme $X$. Then $X(A) \simeq \prod_i X(A_i)$ via the natural map.
\end{theorem}

The result is sharp, as illustrated in Example \ref{ex:algprodschnotqcqs}.

\begin{remark}
	We will prove a version of Theorem \ref{thm:algprodsch} for qcqs algebraic spaces in \S \ref{sec:algprodspaces}. In fact, the proofs are also similar. The main difference is that the case of spaces relies on Theorem \ref{thm:TD}, while, in the world of schemes, we may use Theorem \ref{thm:BC}. A practical consequence is that the proof for schemes is considerably more elementary (but not any more direct) than the proof for spaces.
\end{remark}

Fix $\{A_i\}_{i \in I}$ and $X$ as in the theorem. For proving injectivity of $X(A) \to \prod_i X(A_i)$, we use derived category techniques (as these shall be handy later), though one can also do this directly. The first step is to identify perfect complexes on $A$ as products of (certain) perfect complexes on each $A_i$. More precisely:

\begin{lemma}
	\label{lem:prodfullyfaithful}
	The evident map $\phi:D_\perf(A) \to \prod_i D_\perf(A_i)$ is fully faithful, and $K \simeq \prod_i K \otimes_A A_i$ for any $K \in D_\perf(A)$.
\end{lemma}
\begin{proof}
As $D_\perf(A)$ is generated by $A$ under finite colimits and retracts, for the first claim, it is enough to check that $A \simeq \Hom(\phi(A),\phi(A))$, which is clear. The second claim is proven similarly.
\end{proof}

Next, we prove a special case of Theorem \ref{thm:algprodsch}:

\begin{lemma}
	\label{lem:algprodquasiaffine}
If $X$ is quasi-affine, then $X(A) \simeq \prod_i X(A_i)$.
\end{lemma}
\begin{proof}
	The assertion $X(A) \simeq \prod_i X(A_i)$ is clear if $X$ is affine. Hence, if $X$ is quasi-affine, then one immediately obtains injectivity of $X(A) \to \prod_i X(A_i)$. For surjectivity, fix an affine $Y$ containing $X$ as an open subscheme, with constructible closed complement $Z$. We must show that if $a:\Spec(A) \to Y$ factors through $X$ over each $\Spec(A_i) \subset \Spec(A)$, then $a$ factors through $X$, or, equivalently, that $a^* \calO_Z = 0$. Choose some  $K \in D_Z(Y) \cap D_\perf(Y)$. Then $a^* K = \prod_i (a^* K) \otimes_A A_i \simeq 0$. As $\calO_Z$ can be written as a filtered colimit of such $K$'s, one finds that $a^* \calO_Z = 0$, so $a^{-1}(Z) = \emptyset$, as wanted.
\end{proof}

Using this special case, we can establish injectivity:

\begin{lemma}
	\label{lem:algprodschinj}
	The map $X(A) \to \prod_i X(A_i)$ is injective.
\end{lemma}
\begin{proof}
	Fix two maps $a,b:\Spec(A) \to X$ which induce the same map $a_i = b_i$ over $\Spec(A_i) \subset \Spec(A)$. Now consider the pullback $z:Z \to \Spec(A)$ of $\Delta:X \to X \times X$ along $(a,b):\Spec(A) \to X \times X$. Then $Z \to \Spec(A)$ is a quasi-compact monomorphism as $\Delta$ is so. In particular, $Z$ is quasi-affine. Moreover, $Z$ admits sections over each $\Spec(A_i) \subset \Spec(A)$. Lemma \ref{lem:algprodquasiaffine} gives a unique map $\Spec(A) \to Z$ inducing the given sections over each $\Spec(A_i)$. It follows that $Z \to \Spec(A)$ is a monomorphism with a section, and thus an isomorphism. This immediately gives $a = b$, as wanted.
\end{proof}

\begin{remark}
	\label{rmk:algprodspcinj}
	The proof of Lemma \ref{lem:algprodschinj} applies {\em mutatis mutandis} to qcqs algebraic spaces.
\end{remark}

We now come to the interesting part: the surjectivity of $X(A) \to \prod_i X(A_i)$. Fix maps $f_i:\Spec(A_i) \to X$. We do not know how to directly construct a map $f:\Spec(A) \to X$ extending each $f_i$. Instead, we first define a functor $G:\QCoh(X) \to \Mod_A$ via
\[ G(\calF) := \prod_i \Gamma(\Spec(A_i),f_i^* \calF).\]
The functor $G$ will {\em not} be the pullback functor for the desired map $f:\Spec(A) \to X$. In fact, $G$ does not preserve (infinite) direct sums unless $I$ is finite, so $G$ cannot be a pullback. However, we will later build a new functor $F$ (which will be the desired pullback) from $G$, using crucially the following fact:

\begin{lemma}
	\label{lem:boundschlfp}
	$G$ preserves locally finitely presented objects.
\end{lemma}
\begin{proof}
	Choose affine open covers $\{U_1,\dots,U_r\}$ of $X$ and $\{V_1,\dots,V_m\}$ of $\Spec(A_i)$ such that $f_i$ carries each $V_k$ to some $U_j$, and the numbers $r$ and $m$ are bounded independently of $i \in I$; this is possible by Lemma \ref{lem:boundaffinecover}. We name the index sets $J = \{1,\dots,r\}$ and $K = \{1,\dots,m\}$ for notational simplicity. Choose a locally finitely presented $\calF \in \QCoh(X)$, and write $M := G(\calF) \in \Mod_A$. We must check that $M$ is finitely presented. For each $j \in J$, pick a presentation
	\[ \calO_{U_j}^{\oplus \ell_j} \stackrel{A_j}{\to} \calO_{U_j}^{\oplus n_j} \stackrel{B_j}{\to} \calF|_{U_j} \to 0.\]
	Set $\ell = \max(\ell_j)$ and $n = \max(n_j)$; these are ``absolute'' constants depending only on $X$ and $\calF$. Fix some index $i \in I$, and let $M_i := \Gamma(\Spec(A_i), f_i^* \calF)$. Then $M_i$ is a finite $A_i$-module, and $\widetilde{M_i}|_{V_k}$ is generated by $\leq n$ sections. Lemma \ref{lem:boundgenerators} shows that $\widetilde{M_i}$ is itself generated by $n \cdot m$ sections. This gives a surjection $A_i^{\oplus n \cdot m} \stackrel{Q_i}{\to} M_i$. Note that $n$ and $m$ are independent of the chosen $i \in I$.  Taking products, we get a surjective map $A^{\oplus n\cdot m} \stackrel{Q}{\to} M$, which shows that $M$ is finitely generated.
	
	Let $K_i = \ker(Q_i) \subset A_i^{\oplus n \cdot m}$, and $K = \ker(Q) \subset A^{\oplus n \cdot m}$. As $K = \prod_i K_i$, we must show that $K_i$ is generated by $n'$ elements, for some $n'$ independent of $i$. We will do so by bounding the number of generators for its restriction to each $V_k$. Fix some $k \in K$, and pick $j \in J$ such that $f_i(V_k) \subset U_j$. Set  $L_k := \ker(f_i^* B_j|_{V_k} )$.  Then there is a short exact sequence
	\[ 1 \to L_k \to \calO_{V_k}^{\oplus n_j} \to \widetilde{M_i}|_{V_k} \to 1.\]
	On the other hand, we also have a short exact sequence
	\[ 1 \to \widetilde{K_i}|_{V_k} \to \calO_{V_k}^{\oplus n \cdot m} \to \widetilde{M_i}|_{V_k} \to 1\]
	by definition of $K_i$. Taking a fibre product of the two penultimate maps in these exact sequences, and using that $\Ext^1_{V_k}(\calO_{V_k},-) = 0$ as $V_k$ is affine, one obtains (non-canonically) an isomorphism 
	\[ \calO_{V_k}^{\oplus n \cdot m} \oplus L_k \simeq \calO_{V_k}^{\oplus n_j} \oplus \widetilde{K_i}|_{V_k}.\]
	Note that $L_k$ is generated by $\leq \ell$ global sections as $f_i^*(A_j)|_{V_k}$ factors as 
	\[ \calO_{V_k}^{\oplus \ell_j} \twoheadrightarrow L_k \hookrightarrow \calO_{V_k}^{\oplus n_j},\]
	where the first map is surjective and the second is injective (and recall: $\ell = \max(\ell_j)$).  It follows that $\widetilde{K_i}|_{V_k}$ is generated $\leq N := n\cdot m  + \ell$ sections; note that $N$ is independent of $i$. Another application of Lemma \ref{lem:boundgenerators} shows that $K_i$ is generated $\leq n' := N \cdot m$ elements, as wanted. 
\end{proof}

The following two elementary results were used above. The first bounds the number of generators of a module in terms of local data.

\begin{lemma}
	\label{lem:boundgenerators}
	Fix a ring $R$, and a finite $R$-module $M$.  Assume there exists an integer $n \geq 0$ and an open cover $\{U_1,\dots,U_m\}$ of $\Spec(R)$ such that $\widetilde{M}|_{U_i}$ is generated by $\leq n$ sections. Then $M$ is generated by $\leq n\cdot m$ elements.
\end{lemma}

The proof below was explained to me by de Jong.

\begin{proof}
Set $Y = \Spec(R)$. We work by induction on $m$. There is nothing to prove if $m =1$, so assume $m \geq 2$. Let $Z := Y \setminus U_m$, viewed as a reduced closed subscheme of $Y$. Then $Z$ is an affine scheme covered by $\{V_1,\dots,V_{m-1}\}$, where $V_i = U_i \cap Z$. By the inductive hypotheses, $\widetilde{M}|_Z$ is generated by $\leq n\cdot(m-1)$ sections. Lifting these sections, we can find a map $R^{\oplus n\cdot(m-1)} \to M$ that is surjective on some open neighbourhood $V$ of $Z$ (by Nakayama's lemma). Then $Z' := Y \setminus V$ is a closed subscheme of $Y$, and $\widetilde{M}|_{Z'}$ is generated by $n$ sections (as $Z' \subset U_m$). Lifting a generating set, and adding to the generators found earlier, we obtain a map $R^{\oplus n\cdot m} \simeq R^{\oplus n} \oplus R^{\oplus n \cdot (m-1)} \to M$. This map is surjective over $V \sqcup Z' = \Spec(R)$, which proves the claim.
\end{proof}

The second bounds the number of affines needed to cover a quasi-affine scheme, universally.

\begin{lemma}
	\label{lem:boundaffinecover}
Let $j:W \to Y$ be a quasi-affine morphism of qcqs schemes. Then there exists an integer $m$ such that: for any map $\Spec(B) \to Y$, the pullback $W \times_Y \Spec(B)$ is covered by $m$ affines.
\end{lemma}
\begin{proof}
	By replacing $Y$ with the affine $Y$-scheme $\underline{\Spec}_Y(j_* \calO_W)$,  we may assume $j$ is a quasi-compact open immersion. Let $Z = Y \setminus W$, and choose a locally finitely generated ideal sheaf $I \subset \calO_Y$ defining $Z$. Choose an affine open cover $\{U_1,\dots,U_k\}$ of $Y$ such that $I|_{U_i}$ is generated $\leq r$ global sections (for some $r$), and set $m = k \cdot r$. Fix some map $j:\Spec(B) \to Y$ from an affine scheme. Lemma \ref{lem:boundgenerators} then gives sections $f_1,\dots,f_m \in \Gamma(\Spec(B),j^*I)$ generating $j^*I$. The corresponding distinguished opens $\{D(f_1),\dots,D(f_m)\}$ of $\Spec(B)$ give the desired open cover for $W \times_Y \Spec(B) = \Spec(B)  \setminus j^{-1}(Z)$.
\end{proof}

Write $\QCoh_{fp}(X) \subset \QCoh(X)$ for the full subcategory of finitely presented quasi-coherent sheaves, and similarly for $\Mod_{fp,A} \subset \Mod_A$. Lemma \ref{lem:boundschlfp} shows that  $G$ restricts to a functor $G_{fp}:\QCoh_{fp}(X) \to \Mod_{fp}(A)$. This functor has desirable properties:

\begin{lemma}
	$G_{fp}$ is symmetric monoidal and preserves finite colimits.
\end{lemma}
\begin{proof}
For $\calF_1,\calF_2 \in \QCoh(X)$, there is a natural map $G(\calF_1) \otimes G(\calF_2) \to G(\calF_1 \otimes \calF_2)$; we will first show this map is an isomorphism if $\calF_i \in \QCoh_{fp}(X)$. For this, note that $\Mod_{fp,A} \to \prod_i \Mod_{fp,A_i}$ is fully faithful and symmetric monoidal. Indeed, the latter is automatic, while the former is a consequence of $\Hom_A(M,-)$ commuting with flat base change on $A$ for $M \in \Mod_{fp,A}$. Thus, the assertions for $G_{fp}$ can be checked after composing with the projection $\Mod_{fp,A} \to \Mod_{fp,A_i}$. But the resulting functor $\QCoh(X) \to \Mod_{fp,A_i}$ is simply $\Gamma(\Spec(A_i), f_i^*(-))$, which is clearly symmetric monoidal. The preservation of finite colimits is proven similarly as $\Mod_{fp,A} \to \prod_i \Mod_{fp,A_i}$ preserves finite colimits, and because finite colimits are computed ``termwise'' in the target.
\end{proof}

To build the promised functor $F$, we use a result of Deligne \cite[Appendix, Proposition 2]{HartshorneRD} to identify $\QCoh(X)$ in terms of $\QCoh_{fp}(X)$.

\begin{lemma}[Deligne]
	\label{lem:DeligneIndQC}
	The natural inclusion $\QCoh_{fp}(X) \subset \QCoh(X)$ extends to a symmetric monoidal equivalence $\Ind(\QCoh_{fp}(X)) = \QCoh(X)$ given by $\{A_i\} \mapsto \colim A_i$.
\end{lemma}

We can now prove Theorem \ref{thm:algprod} by applying Theorem \ref{thm:BC}.

\begin{proof}[Proof of Theorem \ref{thm:algprod}]
	The symmetric monoidal functor $G_{fp}:\QCoh_{fp}(X) \to \Mod_{fp,A}$ defines a symmetric monoidal functor $F:\QCoh(X) \to \Mod_A$ by passage to ind-completions. Moreover, both $\QCoh_{fp}(X)$ and $\Mod_{fp,A}$ have finite colimits and $G_{fp}$ preserves these. By formal nonsense, $F$ preserves all colimits. Hence, by Theorem \ref{thm:BC}, there is a unique map $f:\Spec(A) \to X$ such that $F = f^*$. Note that the composition $\QCoh(X) \stackrel{F}{\to} \Mod_A \to \Mod_{A_i}$ is identified with $\Gamma(\Spec(A_i),f_i^*(-))$ as both are cocontinuous and agree on the compact objects $\QCoh_{fp}(X) \subset \QCoh(X)$. Hence, the composition $\Spec(A_i) \subset \Spec(A) \stackrel{f}{\to} X$ induces $f_i^*$ on pullback, and must therefore coincide with $f_i$ by Theorem \ref{thm:BC}. It follows that $f$ is the desired extension.
\end{proof}

\newpage
\section{Algebraization of products: algebraic spaces}
\label{sec:algprodspaces}

The goal of this section is to prove the following theorem.

\begin{theorem}
	\label{thm:algprodspaces}
Fix a set $I$ of rings $\{A_i\}_{i \in I}$ with product $A := \prod_i A_i$, and a qcqs algebraic space $X$. Then $X(A) \simeq \prod_i X(A_i)$ via the natural map.
\end{theorem}

The injectivity of the map appearing in Theorem \ref{thm:algprodspaces} is proven as in Theorem \ref{thm:algprodsch}. For surjectivity, we use Theorem \ref{thm:TD}. To apply this theorem, one must construct a symmetric monoidal functor $F:D_\perf(X) \to D_\perf(\prod_i A_i)$ starting from a family of maps $\{f_i:\Spec(A_i) \to X\}_{i \in I}$. In analogy with the proof of Theorem \ref{thm:algprodsch}, the obvious guess is to use $F(K) = \prod_i \Gamma(\Spec(A_i),f_i^* K) \in D(A)$. In fact, if $F = f^*$ for some map $f:\Spec(A) \to X$, then $F$ is forced to be given by this formula\footnote{This is only true for perfect complexes, not ``large'' quasi-coherent complexes.}. However, in general, it is not clear why $F(K)$ thus defined is a perfect complex: an arbitrary product of perfect complexes $K_i \in D_\perf(A_i)$ is typically not $A$-perfect. Indeed, various numerical invariants associated to the $K_i$'s (such as the cohomological amplitude, the Tor amplitude, the minimal number of generators for $H^0(K_i)$, etc.) might be unbounded as $i$ varies, which immediately precludes $\prod_i K_i$ from being $A$-perfect; see Example  \ref{ex:boundSwanexample} for an explicit example. In \S \ref{ss:boundNisnevich} and \S \ref{ss:boundperfect}, we show that the numerical obstruction is the only one: if the $K_i$'s are presented by a ``bounded amount of data'' (as $i$ varies), then $\prod_i K_i$ is indeed perfect. The phrase ``bounded amount of data'' is made precise by bounding the number and embedding dimension of the projective modules occurring in a presentation for $K_i$ over a Nisnevich cover\footnote{As we work with algebraic spaces, we are forced to use Nisnevich covers instead of Zariski ones.} of $\Spec(A_i)$ of bounded size; the key result, Lemma \ref{lem:boundlocalglobal}, is that a local bound of $N \in \N$ implies a global bound of $f(N) \in \N$ for some function $f:\N \to \N$ (which, crucially, is independent of the ring under consideration). It is then relatively straightforward to check that the functor $F$ defined above does the job, as we do in \S \ref{ss:algprodspacesproof}.

\subsection{Bounding Nisnevich covers}
\label{ss:boundNisnevich}

All algebraic spaces appearing in this subsection are assumed to be qcqs. Recall that a map $f:U \to Y$ of algebraic spaces is called a {\em Nisnevich cover} if it is an \'etale cover that admits sections over a constructible stratification of $Y$ (see \cite[\S 1]{LurieDAGXI}). We will need a quantitative variant:

\begin{definition}
	A Nisnevich cover $f:U \to Y$ of an algebraic space $Y$ has {\em length $\leq m$} if there exists a flag $\emptyset = Z_0 \subset Z_1 \subset \dots \subset Z_m := Y$ of constructible closed subsets such that $f|_{Z_i \setminus Z_{i-1}}$ has a section for all $i$.
\end{definition}

Note that every Nisnevich cover $f:U \to Y$ has $\leq m$ for some $m > 0$. Moreover, this property is stable under base change.

\begin{lemma}
Fix a Nisnevich cover $f:U \to Y$ of an algebraic space $Y$ of length $\leq m$. For any map $g:Y' \to Y$, the base change $U \times_Y Y' \to Y'$ of $f$ along $g$ is a Nisnevich cover of length $\leq m$.
\end{lemma}
\begin{proof}
As $Y'$ and $Y$ are qcqs, the map $g$ is a qcqs map, and thus $g^{-1}$ preserves constructibility, which immediately gives the lemma.
\end{proof}

Lengths multiply under compositions.

\begin{lemma}
	Fix a composite $Y_1 \stackrel{g_1}{\to} Y_2 \stackrel{g_2}{\to} Y_3$ of Nisnevich covers of algebraic spaces such that $g_i$ has length $\leq m_i$ for $i = \{1,2\}$ and suitable $m_i \in \N$. Then $g_2 \circ g_1$ has length $\leq m_1 \cdot m_2$.
\end{lemma}
\begin{proof}
	Let $Z_\bullet := \{\emptyset = Z_0 \subset Z_1 \subset \dots \subset Z_{m_2} = Y_3\}$ and $W_\bullet := \{\emptyset = W_0 \subset W_1 \subset \dots \subset W_{m_1} = Y_2\}$ be flags of constructible closed subsets witnessing the lengths of $g_2$ and $g_1$ respectively. We will inductively construct a length $m_2$ flag $K^i_\bullet$ of constructible closed subsets $K^i_j \subset Z_i$ such that $K^i_j$ contains $Z_{i-1}$ for all $j$, and $g_2 \circ g_1$ admits sections over each $K^i_j \setminus K^i_{j-1}$. Putting the various $K^i_\bullet$ together then gives a flag in $Y_3$ of size $m_1 \cdot m_2$ with the required properties. For $i = 1$, we simply use $K^1_\bullet := W_\bullet \cap Z_1$, where $Z_1$  is viewed as a closed subset of $Y_2$ via some chosen section of $g_2$ over $Z_1$. Assume such flags have been constructed for $Z_i$. Then $Z_{i+1} \setminus Z_i$ may be viewed as a subscheme of $Y_2$ via some chosen section. Thus, $K^{i+1,'}_\bullet := W_\bullet \cap (Z_{i+1} \setminus Z_i)$ defines a flag of constructible closed subsets of $Z_{i+1} \setminus Z_i$ of length $m_2$ such that $g_2 \circ g_1$ admits sections over $K^{i+1,'}_j \setminus K^{i+1,'}_{j-1}$. Setting $K^{i+1}_\bullet = K^{i+1,'} \cup Z_i$ gives the desired flag.
\end{proof}

As a result, lengths behave predictably under Zariski covers.

\begin{example}
	Let $f:U \to Y$ be a Nisnevich cover of algebraic spaces with length $\leq m$. Fix an open cover $\{U_1,\dots,U_k\}$ of $U$. Then the composite $\sqcup U_i \to Y$ is a Nisnevich cover of length $\leq m \cdot k$. To see this, it is enough to check that $g:\sqcup U_i \to U$ is a Nisnevich cover of length $\leq k$. We show this by induction on $k$. If $k = 1$, the claim is clear. In general, set $Z_k := U$ and $Z_{k-1} := U \setminus U_1$. Note that $\sqcup U_i \to U$ has a section over $U_1 := Z_k \setminus Z_{k-1}$. As $U_1 \cap Z_{k-1} = \emptyset$, the inductive hypotheses applies to the restriction of $g$ to $Z_{k-1}$ to give a flag $\emptyset = Z_0 \subset Z_1 \subset \dots \subset Z_{k-1}$ of closed subschemes such that $g$ admits sections over $Z_i \setminus Z_{i-1}$. It is then clear that $g$ has length $\leq k$.
\end{example}

Using lengths, one can bound the minimal number of generators of a module over a ring in terms of the corresponding number over a Nisnevich cover, in analogy with Lemma \ref{lem:boundgenerators}.

\begin{lemma}
	\label{lem:boundgeneratorsNisnevich}
	Fix a Nisnevich cover $f:U \to Y$ of length $\leq m$ of an affine scheme $Y = \Spec(R)$, and some $M \in \Mod_R$. If $f^* \widetilde{M}$ is generated by $\leq N$ global sections, then $\widetilde{M}$ is generated by $\leq Nm$ global sections.
\end{lemma}
\begin{proof}
	We prove the claim by induction on $m$.  Choose a flag $\emptyset = Z_0 \subset Z_1 \subset \dots Z_m := Y$ of constructible closed subsets of $Y$ such that $f$ admits sections over $Z_i \setminus Z_{i-1}$. If $m = 1$, then $Z_1 = Y$, so there is nothing to show. For $m > 1$, choose a map $\phi:R^{\oplus m} \to M$ that is surjective over $Z_1$; this is always possible as $\widetilde{M}|_{Z_1}$ is generated by $\leq m$ sections by the assumption on $f$, and because $Z$ and $X$ are affine. The cokernel $Q = \coker(\phi)$ is a finitely presented $R$-module whose support is a closed subscheme $W := \Spec(R/I) \subset \Spec(R)$ that does not intersect $Z_1$. The restriction $f|_{W}$ is then a Nisnevich cover of length $\leq m-1$. As $Q$ is a quotient of $M/IM$, one finds, by induction, a surjection $(R/I)^{\oplus N \cdot (m-1)} \to Q$. Lifting sections to $R$ and $M$ gives a map $\psi:R^{\oplus N \cdot (m-1)} \to M$ whose image surjects onto $Q$. The sum $\phi \oplus \psi:R^{\oplus N \cdot m} \to M$ is then easily seen to be surjective, proving the claim.
\end{proof}

One also has an analogue of Lemma \ref{lem:boundaffinecover}.

\begin{lemma}
	\label{lem:boundaffinecoverNisnevich}
Let $j:W \to Y$ be a quasi-affine morphism of algebraic spaces. Then there exists an integer $m$ such that: for any map $\Spec(B) \to Y$, the pullback $W \times_Y \Spec(B)$ is covered by $m$ affines.
\end{lemma}
\begin{proof}
	By replacing $Y$ with the affine $Y$-space $\underline{\Spec}_Y(j_* \calO_W)$,  we may assume $j$ is a quasi-compact open immersion. Let $Z = Y \setminus W$, and choose a locally finitely generated ideal sheaf $I \subset \calO_Y$ defining $Z$. Choose a Nisnevich cover $f:U \to Y$ of length $\leq k$ such that $f^* I$ is generated $\leq r$ global sections. Set $m = k \cdot r$. Fix some map $j:\Spec(B) \to Y$ from an affine scheme. Then $W \times_Y \Spec(B) \to \Spec(B)$ is a quasi-compact open immersion defined by the ideal $j^* I$. Lemma \ref{lem:boundgeneratorsNisnevich}, applied to the pullback of $f$ along $j$, shows that $j^*I$ is defined by $m$ global sections $f_1,\dots,f_m \in \Gamma(\Spec(B),j^*I)$. The corresponding distinguished opens $\{D(f_1),\dots,D(f_m)\}$ of $\Spec(B)$ give the desired open cover for $W \times_Y \Spec(B)$.
\end{proof}

We will actually need a more precise version of a special case of Lemma \ref{lem:boundaffinecoverNisnevich}:

\begin{lemma}
	\label{lem:boundaffinepullback}
	Let $f:U \to Y$ be a Nisnevich cover of an algebraic space $Y$ of length $\leq m'$. There exists an integer $m$ such that: for any map $\Spec(B) \to Y$, there exists a Nisnevich cover $V \to \Spec(B)$ of length $\leq m$ with $V$ affine, and a $Y$-map $V \to U$.
\end{lemma}
\begin{proof}
	Lemma \ref{lem:boundaffinecoverNisnevich} applied to $f$ gives an integer $m''$ such that $V' := U \times_Y \Spec(B)$ admits an affine open cover $\{V'_1,\dots,V'_{m''}\}$. Set $V = \sqcup V'_i$, so the map $V \to V'$ is a Nisnevich cover of length $\leq m''$. Setting $m = m' \cdot m''$ then solves the problem.
\end{proof}

\subsection{Bounding perfect complexes}
\label{ss:boundperfect}

The goal of this subsection is to formulate and prove the promised result bounding the size of the presentation of a perfect complex over an affine scheme in terms of the same data over a bounded Nisnevich cover. We make the following {\em ad hoc} definitions for the ``size'':

\begin{definition}
	Fix a ring $R$, an object $K \in D_\perf(R)$, and some positive integer $N$.  We say:
	\begin{enumerate}
		\item $K$ {\em locally has size $\leq N$} if there exists a Nisnevich cover $f:U \to \Spec(R)$ of length $\leq N$ with $U$ affine such that $f^* \widetilde{K}$ is represented by a finite complex $P^\bullet$ of finite locally free $\calO_U$-modules $P^i$ with $P^i = 0$ for $|i| > N$ and each $P^i$ being a retract of $\calO_U^{\oplus N}$.
		\item $K$ {\em globally has size $\leq N$} if $K$ can be represented as a complex $P^\bullet$ of finite projective $R$-modules $P^i$ with $P^i = 0$ for $|i| > N$ and each $P^i$ being a retract of $R^{\oplus N}$.
	\end{enumerate}
\end{definition}

We record an elementary property of these notions.

\begin{lemma}
For any ring $R$, each $K \in D_\perf(R)$ globally has size $\leq N$ for some $N > 0$. Moreover, in this case, $K$ locally has size $\leq N$.
\end{lemma}
\begin{proof}
Clear from the definition.
\end{proof}

The key result is a converse to the previous lemma: one may propagate local bounds to global ones.

\begin{lemma}
	\label{lem:boundlocalglobal}
There exists a function $f:\N \to \N$ such that: for every ring $R$ and $K \in D_\perf(R)$, if $K$ locally has size $\leq N$, then $K$ globally has size $\leq f(N)$.
\end{lemma}

The proof below constructs an $f$ with $f \sim O(N^{2^{4N}} )$; is it optimal?

\begin{proof}
	The proof involves a nested induction. The ``outer'' induction is along the cohomological amplitude, while the ``inner'' induction is along the Tor amplitude. As both these quantities are bounded by twice the local size, both inductions are actually finite. 
	
	More precisely, we will recursively construct functions $f_i:\N \to \N$ for $i = 0,\dots,2N$ such that:
	\begin{enumerate}[(1)]
		\item For any ring $R$ and $K \in D_\perf(R)$, if $K$ locally has size $\leq N$ and cohomological amplitude of size $\leq i$, then $K$ globally has size $\leq f_i(N)$.
	\end{enumerate}
	One then defines $f(N) = f_{2N}(N)$. This clearly does the job because any $K \in D_\perf(R)$ which locally has size $N$ has cohomological amplitude of size $\leq 2N$.

	Assume first $i = 0$. We will recursively construct functions $f_0^j:\N \to \N$ for $j=0,\dots,2N$ such that:
	\begin{enumerate}[(a)]
		\item For any ring $R$ and $K \in D_\perf(R)$ with cohomological amplitude $0$, if $K$ locally has size $\leq N$ and Tor amplitude of size $\leq j$, then $K$ globally has size $\leq f_0^j(N)$. 
	\end{enumerate}
	Taking $f_0 = f_0^{2N}$ then solves the problem of constructing $f_0$ as any $K \in D_\perf(R)$ which locally has size $\leq N$ also has Tor amplitude $\leq 2N$. For $j=0$, we have:
	
	\begin{claim} 
		$f_0^0(N) = N^2$ satisfies (b) for $j = 0$.
	\end{claim}
	\begin{proof}[Proof of Claim] Fix some $K \in D_\perf(R)$ which locally has size $\leq N$ (for some $N$), and cohomological and Tor amplitude of size $0$. Then $K = M[k]$ for a finite projective $R$-module $M$ and some $k \in \Z$. Fix a Nisnevich cover $f:U \to X$ of length $\leq N$ with $U$ affine, as well as a finite complex $P^\bullet$ of finite locally free $\calO_U$-modules as in the definition of local size. As $\tau^{> k}(K) = 0$, one checks (using that $U$ is affine) that $Z^k(P^\bullet)$ is a finite projective $\calO_U$-module which occurs as a retract of $\calO_U^{\oplus N}$. Moreover, $Z^k(P^\bullet) \to \calH^k(f^* \widetilde{K} )$ is surjective by definition. Lemma \ref{lem:boundgeneratorsNisnevich} then shows that $\widetilde{M} = \calH^k(K)$ is generated by $N^2$ global sections. As $M$ is projective, one may realize $M$ as a summand of $R^{\oplus N^2}$, which proves the claim.
	\end{proof}

	To construct the remaining $f_0^j$ inductively, fix some $0 < j \leq 2N$ and assume that $f_0^{j'}$ has been constructed for $j' < j$. It suffices to show: for any ring $R$ and any non-zero $K \in D_\perf(R)$ which locally has size $\leq N$, cohomological amplitude $0$, and Tor amplitude $\leq j$, there exists a cofibre sequence
	\[ K' \to  R^{\oplus N^2}[-k] \stackrel{h}{\to} K\]
where $k$ is the unique integer such that $H^k(K) \neq 0$, and $h$ is surjective on $H^k$. Indeed, then $K' \in D_\perf(R)$ locally has size $\leq N^2 + N$ (by the formula for a mapping cone), cohomological amplitude $0$, and Tor amplitude $\leq j-1$, so one may use $f_0^j(N) = N^2 + f_0^{j-1}(N^2 + N)$. The map $h$ can be constructed using the technique from the previous paragraph, so we have constructed $f_0$ satisfying (1).

We now construct $f_i$ for $i > 0$. Fix some $0 < i \leq 2N$, and assume that $f_j:\N \to \N$ satisfying (1) have been constructed for $j < i$. We claim that $f_i(N) = N^2 + f_{i-1}(N^2 + N)$ does the job. For this, fix some ring $R$ and $0 \neq K \in D_\perf(R)$ with local size $\leq N$ and cohomological amplitude of size $\leq i$. It is enough to construct a cofibre sequence
	\[ R^{\oplus N^2}[-k] \stackrel{g}{\to} K \to  Q \]
where $k$ is the largest integer with $H^k(K) \neq 0$, and $g$ is surjective on $H^k$. Indeed, once such a cofibre sequence is constructed, $Q$ locally has size $\leq N^2 + N$ (by the formula for the mapping cone) and cohomological amplitude of size $\leq i-1$, so $K$ globally has size $\leq N^2 + f_{i-1}(N^2 + N)$ (by the formula for the mapping cone). To construct this cofibre sequence,  it is enough to construct $g$. Choose $k$ as above. Then $H^k(K)$ can be generated by $\leq N^2$ elements by the argument used in the previous two paragraphs. This gives a surjective map $\overline{g}:R^{\oplus N^2} \to H^k(K)$. Then $\cofib(K[-k] \to H^k(K))$ is $1$-connected, so we may lift $\overline{g}$ to a map $g[-k]:R^{\oplus N^2} \to K[-k]$ that is surjective on $H^0$; shifting by $k$ gives the desired map.
\end{proof}

\subsection{Proof of Theorem \ref{thm:algprodspaces}}
\label{ss:algprodspacesproof}

Fix $X$, $\{A_i\}$, and $A$ as in Theorem \ref{thm:algprodspaces}. The proof of injectivity of $X(A) \to \prod_i X(A_i)$ proceeds exactly as before. For surjectivity, we first show that a family of projective modules over each $A_i$ with uniformly bounded embedding dimension can be patched to a projective module over $A$.

\begin{lemma}
	\label{lem:boundprojectivepatch}
	Fix projective $A_i$-modules $P_i$. Assume $P_i$ is a retract of $A_i^{\oplus N}$ for some $N$ independent of $i$. Then $P := \prod_i P_i$ is a projective $A$-module, and a retract of $A^{\oplus N}$.
\end{lemma}
\begin{proof}
Choose projectors $\epsilon_i \in M_N(A_i)$ realizing $P_i$, i.e., $\epsilon_i^2 = \epsilon_i$ and $P_i = \im(\epsilon_i)$. Then $\epsilon = \prod_i \epsilon_i \in M_N(A)$ is a projector, and $P = \im(\epsilon)$ is indeed projective; we use here that the formation of cokernels commutes with arbitrary products in abelian groups.
\end{proof}

We can upgrade this to a patching result for perfect complexes.

\begin{lemma}
	\label{lem:boundglobalpatch}
Fix $K_i \in D_\perf(A_i)$ which globally have size $\leq N$ for some $N$ independent of $i$. Then $K := \prod_i K_i \in D_\perf(A)$.
\end{lemma}
\begin{proof}
	We may represent each $K_i$ by a finite complex $P^\bullet_i$ of finite projective $A_i$-modules such that $P_i = 0$ for $|i| > N$ and $P_i$ is a retract of $A_i^{\oplus N}$. Lemma \ref{lem:boundprojectivepatch} shows that $P^\bullet := \prod_i P_i^\bullet \in D(A)$ has finite projective terms with $P^i = 0$ for $|i| > N$, so $K \simeq P^\bullet$ is perfect.
\end{proof}

Using this patching result and the bounds from \S \ref{ss:boundperfect}, we can finishing proving surjectivity.

\begin{proof}[Proof of Theorem \ref{thm:algprodspaces}]
	We have already seen that $X(A) \to \prod_i X(A_i)$ is injective (see Remark \ref{rmk:algprodspcinj}). For surjectivity, fix maps $f_i:\Spec(A_i) \to X$ for $i \in I$. We claim that the association $K \mapsto \prod_i \Gamma(\Spec(A_i),f_i^* K)$ defines a functor $F':D(X) \to D(A)$ that preserves perfect complexes. Fix some $K \in D_\perf(A)$. Then $K$ locally has size $\leq N$ for some $N$; here we implicitly use that $X$ has a Nisnevich cover by affine schemes (see \cite[Tag 08GL]{StacksProject} or \cite[Theorem 1.3.8]{LurieDAGXII}). Using Lemma \ref{lem:boundaffinepullback}, one checks that $f_i^* \widetilde{K}$ locally has size $\leq N'$ for some $N'$ independent of $i$.  By Lemma \ref{lem:boundlocalglobal}, the complex $f_i^* \widetilde{K}$ globally has size $\leq N''$ for some $N''$ independent of $i$. By Lemma \ref{lem:boundglobalpatch}, it follows that $F'(K)$ is perfect. Using Lemma \ref{lem:prodfullyfaithful}, one easily checks that the resulting functor $F:D_\perf(X) \to D_\perf(A)$ is symmetric monoidal. By Theorem \ref{thm:TD}, one obtains a map $f:\Spec(A) \to X$ such that $f^* = F$. As the composition of $F$ with any projection $D_\perf(A) \to D_\perf(A_i)$ is simply $f_i^*$, it follows that $f$ extends each $f_i$, as wanted.
\end{proof}

\newpage
\section{Some examples and applications}
\label{sec:algprodex}

The main goal of this section is to record some special cases of Theorem \ref{thm:algprodsch} and Theorem \ref{thm:algprodspaces} that can be proven by hand. Along the way, we also give counterexamples illustrating the sharpness of these results.  We begin with two examples where Theorem \ref{thm:algprodsch} can be proven directly.

\begin{example}
	\label{ex:algprodZlocal}
	With notation as in Theorem \ref{thm:algprodsch}, assume that each $A_i$ is local. Choose a Zariski cover $\{U_1,\dots,U_n\}$ of $X$, and set $U = \sqcup_j U_i$. Then every map $\Spec(A_i) \to X$ lifts to $U$, so the surjectivity of $X(A) \to \prod_i X(A_i)$ follows from that for $U$. The injectivity is proven as before. \end{example}

\begin{example}
	\label{ex:algprodprojective}
	Let $X = \P^n$, and fix a set $I$ of rings $\{A_i\}_{i \in I}$. Write $A = \prod_i A_i$. We will show $X(A) = \prod_i X(A_i)$ by interpreting $X(R)$ as the collection of tuples $(L,s_0,\dots,s_n)$ where $L \in \Pic(R)$ and $s_i \in L$ such that $R^{\oplus n+1} \stackrel{s_i}{\to} L$ is surjective (up to isomorphism). By suitably twisting, one first checks that $\underline{\Pic}(A) \to \prod_i \underline{\Pic}(A_i)$ is fully faithful; here $\underline{\Pic}(-)$ denotes the Picard {\em category} functor. It is then relatively easy to see that $X(A) \to \prod_i X(A_i)$ is injective. For surjectivity, one must show: given $x_i := (L_i,s_{i,0},\dots,s_{i,n} ) \in X(A_i)$, there exists $x := (L,s_0,\dots,s_n) \in X(A)$ inducing $x_i$. This follows by the argument in Lemma \ref{lem:boundprojectivepatch}. A similar argument works whenever $X$ is quasi-projective over an affine (using the trick from Lemma \ref{lem:algprodquasiaffine} ).
\end{example}

The argument in Example \ref{ex:algprodprojective} (or, rather, Lemma \ref{lem:boundprojectivepatch} ) crucially relies on bounding the embedding dimension of the line bundles. This is necessary: an arbitrary product of line bundles $L_i \in \Pic(A_i)$ does not give a line bundle on $A$, and thus Theorem \ref{thm:algprodsch} does not extend to Artin stacks, as the next example shows.

\begin{example}
	\label{ex:boundSwanexample}
	Let $X = B\G_m$, and fix a set $\{A_i\}$ of rings. Then $X(\prod_i A_i) \neq \prod_i X(A_i)$ in general. More precisely, the natural map $X(\prod_i A_i) \to \prod_i X(A_i)$ is not essentially surjective. To see this, it is enough to exhibit rings a sequence $\{A_n\}$ of rings with line bundles $M_n \in \Pic(A_n)$ such that $M_n$ is generated by no fewer than $f(n)$ sections, where $f:\N \to \N$ is an unbounded function; this is simply because any line bundle on $\Spec(\prod_n A_n)$ defines a line bundle on $\Spec(A_n)$ generated by $N$ sections for some $N \gg 0$ independent of $n$. Such a sequence of line bundles $\{M_n\}$ and rings $\{A_n\}$ was constructed by Swan (with $A_n$ noetherian); see \cite{SwanVectorBundles}.
\end{example}

The next example shows that Theorem \ref{thm:algprodsch} fails for the simplest Deligne-Mumford stacks; the underlying reason is the purely topological fact that the classifying space of a finite group is infinite dimensional, though we argue cohomologically in the algebraic context.

\begin{example}
	Fix an algebraically closed field $k$ of characteristic $0$, and let $G$ be a non-trivial finite group. Let $X = BG$ be the classifying stack of $G$-torsors.  We will construct affine schemes $\Spec(A_i)$ for each $i \in \N$, and maps $f_i:\Spec(A_i) \to X$ such that $\sqcup_i \Spec(A_i) \to X$ does not factor through a map $\Spec(\prod_i A_i) \to X$.  For the construction, fix a prime $p$ dividing the order of $G$, so $H^i_\et(X,\Z/p) \neq 0$ for arbitrarily large $i \in \N$\footnote{In fact, by \cite{QuillenGroupCoh}, one knows that $H^*_\et(X,\Z/p)/(\mathrm{nilpotents})$ is a finitely generated $\Z/p$-algebra of non-zero Krull dimension, and thus $H^i_\et(X,\Z/p)$ cannot be $0$ for all $i \gg 0$.}. For each $i \in \N$, choose an affine variety $U_i := \Spec(A_i)$ and a map $U_i \to X$ which is an isomorphism on $H^{\leq i}_\et(-,\Z/p)$; this can always be done by looking at the quotient by $G$ of the stabilizer-free locus in a sufficiently large representation of $G$ and using the Jounalou trick; see \cite[Lemma 1.6]{TotaroChowRingBG} or \cite[\S 4.2]{MorelVoevodsky}. Consider the resulting map $f_i:\Spec(A_i) \to X$. We claim that $\sqcup_i f_i:\sqcup U_i \to X$ does not factor through some map $f:\Spec(\prod_i A_i) \to X$. Assume towards contradiction such an $f$ does exist. As $X$ is locally finitely presented, we can find a finitely presented $k$-subalgebra $A \subset \prod_i A_i$ such that $f$ factors through some map $\alpha: \Spec(A) \to X$. As $\Spec(A)$ is an algebraic variety, one knows $H^k_\et(\Spec(A),\Z/p) = 0$ for $k > \dim(A)$ by Artin vanishing. In particular, it follows that $f_i$, viewed as the composite map $\Spec(A_i) \to \Spec(\prod_i A_i) \to \Spec(A) \to X$, induces the $0$ map on $H^{> \dim(A)}_\et(-,\Z/p)$ for all $i$. However, for $i \gg 0$, the map $f_i$ induces a non-zero map on $H^k(-,\Z/p)$ for some $k > \dim(A)$ by construction, which gives the desired contradiction.
\end{example}

The next example shows that Theorem \ref{thm:algprodsch} is false if $X$ is not qc or not qs.

\begin{example}
	\label{ex:algprodschnotqcqs}
	Take an infinite set $I$ and set $A_i := k$ for some non-zero ring $k$. If $X = \sqcup_i \Spec(A_i)$ is the displayed non-quasi-compact scheme, it is easy to see that $X(\prod_i A_i) \to \prod_i X(A_i)$ is not surjective. Now set $Y$ to be the glueing of $\overline{X} := \Spec(\prod_i A_i)$ to itself along the identity on $X \subset \overline{X}$. Then $Y$ is quasi-compact, but not quasi-separated. It is easy to see that the two distinct maps $\overline{X} \to Y$ induce the same map $\Spec(A_i) \to Y$, so $Y(\prod_i A_i) \to \prod_i Y(A_i)$ is not injective.
\end{example}

The next example contains a direct proof of an important special case of Theorem \ref{thm:algprodspaces}, and is closely related to Example \ref{ex:algprodZlocal}.

\begin{example}
	\label{ex:adelicpoints}
Let $X$ be a qcqs algebraic space, and assume $\{A_i\}$ is a set of henselian local rings. Set $A = \prod_i A_i$. Then one can show $X(A) \simeq \prod_i X(A_i)$ directly as follows. The argument for injectivity in the proof of Theorem \ref{thm:algprodspaces} is elementary,  and we offer no improvements here. For surjectivity, fix a Nisnevich cover $U \to X$ with $U$ an affine scheme. Given maps $a_i:\Spec(A_i) \to X$, one may choose lifts $\widetilde{a_i}:\Spec(A_i) \to U$ as $A_i$ is henselian local. This shows $\prod_i U(A_i) \to \prod_i X(A_i)$ is surjective. As $U$ is affine, one clearly has $U(A) = \prod_i U(A_i)$, so the composite $U(A) \to X(A) \to \prod_i X(A_i)$ is surjective, and hence $X(A) \to \prod_i X(A_i)$ is surjective. 
\end{example}

We discuss one application of  Theorem \ref{thm:algprod} to describing adelic points on algebraic spaces over global fields; in fact, only the significantly easier Example \ref{ex:adelicpoints} is used the proof, but we record the statement anyways. First, we fix some notation (and adhere to standard conventions in number theory for any unexplained notation). Let $K$ be a global field, $S$ a finite non-empty set of places of $K$ (assumed to contain the places at $\infty$), $\A_K$ the adele ring of $K$, and $\A_{K,S} \subset \A$ the subring of adeles integral outside $S$. Then we have:

\begin{corollary}
	\label{cor:adelicpoints}
	For any qcqs algebraic space $X$ over $\calO_{K,S}$, the natural map induces bijections 
	\[ X(\A_{K,S} ) \simeq \prod_{v \in S} X(K_v) \times \prod_{v \notin S} X(\calO_v).\] 
	If additionally $X$ is finitely presented over $\calO_{K,S}$, then
	\[ X(\A_K) \simeq \prod_{v \in S} X(K_v) \times {\prod_{v \notin S}}^{'}  \Big(X(K_v),X(\calO_v) \Big).\]
\end{corollary}
In the special case $X = \G_a$, Corollary \ref{cor:adelicpoints} is a definition. Slight variants of Corollary \ref{cor:adelicpoints} can also be found in work of Conrad \cite[page 613-615]{ConradFinitenessAlgGroups} and \cite[Theorem 3.6]{ConradAdelicPoints}.
\begin{proof}
	The first part is immediate from Example \ref{ex:adelicpoints} as $\calO_v$ and $K_v$ are henselian local rings. For the second, note that $\A_K = \colim \A_{K,T}$, where the colimit runs over finite sets $T$ of places containing $S$. As $X$ is finitely presented, one obtains $X(\A_K) = \colim X(\A_{K,T})$. By definition of the restricted product, one also has 
	\[ \prod_{v \in S} X(K_v) \times {\prod_{v \notin S}}^{'}  \Big(X(K_v),X(\calO_v) \Big) \simeq \colim \Big(\prod_{v \in S} X(K_v) \times \prod_{v \in T \setminus S} X(K_v) \times \prod_{v \notin T} X(\calO_v)\Big),\]
	where the colimit is indexed by the same $T$'s as before. The claim is now immediate from the first part.
\end{proof}

\newpage

\bibliography{algebraise}
\end{document}